\documentclass[11pt]{article}
\usepackage{hyperref}
\usepackage{amsmath}
\usepackage{amsthm} 
\usepackage{amsfonts}
\usepackage{graphicx} 
\usepackage{latexsym}
\usepackage{url}
\usepackage{color}
\usepackage{caption} 
\usepackage{subcaption} 
\usepackage{graphicx}

\newcommand{\R}{\mathbb{R}}
\newcommand{\E}{\mathbb{E}}
\newcommand{\PP}{\mathbb{P}}

\newcommand{\al}{\alpha}

\newcommand{\ga}{\gamma}

\newcommand{\ep}{\varepsilon}

\newcommand{\om}{\omega}

\newcommand{\m}{\mathcal{M}}

\newcommand{\uu}{\underline{u}}

\def\nib{\noindent\bf}

\newcommand{\vsq}{\vskip 4mm}

\newcommand{\stab}{\stackrel{d_{st}}{\longrightarrow}}

\newcommand{\toop}{\stackrel{\PP}{\longrightarrow}}

\newcommand{\Var}{\mbox{\rm Var}}

\newcommand{\diag}{\mbox{\rm diag}}

\newcommand{\bee}{\begin{equation}}
\newcommand{\eee}{\end{equation}}
\newcommand{\bea}{\begin{eqnarray}}
\newcommand{\eea}{\end{eqnarray}}
\newcommand{\bean}{\begin{eqnarray*}}
\newcommand{\eean}{\end{eqnarray*}}

\renewcommand{\qed}{$\hfill\Box$}

\newcommand{\I}{\mathbf{I}}
\newcommand{\M}{\mathcal{M}}
\newcommand{\rank}{\mbox{\rm rank}}
\newcommand{\sgn}{\mbox{\rm sgn}}

\newtheorem{prop}{Proposition}[section]
\newtheorem{cor}[prop]{Corollary}

\newtheorem{lem}[prop]{Lemma}
\newtheorem{ex}[prop]{Example}
\newtheorem{theo}[prop]{Theorem}
\newtheorem{rem}[prop]{Remark}

\allowdisplaybreaks 

\begin{document}

\title{Testing the maximal rank of the volatility process for continuous diffusions observed with noise}
\author{ Tobias Fissler \thanks{University of Bern, Department of Mathematics and Statistics, Institute of Mathematical Statistics and Actuarial Science, Sidlerstrasse 5, 3012 Bern, Switzerland, Email: tobias.fissler@stat.unibe.ch} \thanks{This author is supported by Swiss National Science Foundation.} \and
Mark Podolskij \thanks{Department of Mathematics, University of Aarhus,
Ny Munkegade 118, 8000 Aarhus C,
Denmark, Email: mpodolskij@creates.au.dk.}}

\date{\today}

\maketitle

\begin{abstract}
In this paper, we present a test for the maximal rank of the volatility process in continuous diffusion models
observed with noise. Such models are typically applied in mathematical finance, where latent price processes are corrupted by
microstructure noise at ultra high frequencies. Using high frequency observations we construct a test statistic for the maximal
rank of the time varying stochastic volatility process. Our methodology is based upon a combination of a matrix perturbation approach 
and pre-averaging. We will show the asymptotic mixed normality of the test statistic and obtain a consistent testing procedure.

\ \

\bigskip


\end{abstract}

\section{Introduction}
\label{sec1}
\setcounter{equation}{0}
\renewcommand{\theequation}{\thesection.\arabic{equation}}

In the last twenty years, asymptotic theory for high frequency data has received a great deal of attention
in probability and statistics. This is mainly motivated by financial applications, where observations
of stock prices are recorded  very  frequently. In an ideal world, i.e. under no-arbitrage conditions, price processes must follow an It\^o
semimartingale, which is a celebrated result of Delbaen and Schachermayer \cite{DS94}. We refer to a monograph \cite{JP12} for a comprehensive
study of limit theorems for It\^o semimartingales and their manifold applications in statistics. 

Despite the aforementioned theoretical result, at ultra high frequencies, the financial
data is contaminated by {\it microstructure noise} such as rounding errors, bid-ask bounces and misprints. One of the standard models
for the microstructure noise is an additive i.i.d. process independent of the latent price (see e.g. \cite{BHLS08,ZMA05} among many others; an extension 
of this model can be found in \cite{JLMPV09}). More formally, the model is given as
\begin{align} \label{diffusion}
Y_{t_i} = X_{t_i} + \varepsilon_{t_i} \qquad \text{with} \qquad dX_t= b_t dt + \sigma_t dW_t, 
\end{align}   
where $(X_t)_{t\in [0,T]}$ is a $d$-dimensional continuous It\^o semimartingale, $t_i=i\Delta_n$ and $(\varepsilon_t)_{t\in [0,T]}$ is a 
$d$-dimensional i.i.d process
independent  of $X$ with  
\begin{align} \label{noise}
\E[\varepsilon_t]=0 \quad \text{and} \quad \E[\varepsilon_t \varepsilon_t^{\star}]=:\Sigma \in \R^{d \times d}. 
\end{align}
We are in the framework of infill asymptotics, i.e. $\Delta_n\rightarrow 0$ while $T$ remains fixed. This paper
is devoted to the test for the maximal rank of the co-volatility matrix $c_t=\sigma_t\sigma_t^{\star}$ of the unobserved 
diffusion process $X$. We remark that this is an equivalent formulation of the following problem: What is the minimal amount of independent
Brownian motions required for modeling the $d$-dimensional diffusion $X$? Answering this question might give a direct economical interpretation
of the financial data at hand. Furthermore, testing for the full rank of $c_t$ is connected to testing for completeness of financial markets.   

In a recent paper \cite{JP13}, the described statistical problem has been solved in a continuous diffusion setting without noise (we also 
refer to an earlier article \cite{JLT08} for a related problem). The main idea is based upon a  matrix perturbation method, which helps to identify
the rank of a given matrix. The maximal rank of the stochastic co-volatility process $(c_t)_{t\in [0,T]}$ is then asymptotically identified via 
a certain ratio statistic, which uses the scaling property of a Brownian motion. Clearly, the test statistic becomes invalid in the framework
of continuous diffusion models observed with noise. To overcome this problem we apply the pre-averaging approach, which has been originally proposed in 
\cite{JLMPV09,PV09}. As the name suggests,  weighted averages of increments of the process $Y$ are built over a certain window in order to eliminate
the influence of the noise to some extent. This in turn gives the possibility to infer the co-volatility process $(c_t)_{t\in [0,T]}$. The size of the 
pre-averaging window $k_n$
is typically chosen as $k_n=O(\Delta_n^{-1/2})$ and objects as the integrated co-volatility $\int_0^T c_t dt$ can be estimated with the convergence rate
of $\Delta_n^{-1/4}$, which is known to be optimal.

At this stage, we would like to stress that combining the pre-averaging approach and the matrix perturbation method is by far not trivial. There are mainly
two problems that need to be solved. First of all, when using the optimal window size of $k_n=O(\Delta_n^{-1/2})$ in the pre-averaging approach, the diffusion
and the noise parts have the same order, and it becomes virtually impossible to distinguish the rank of the co-volatility from the unknown rank of the covariance 
matrix $\Sigma$. Hence, we will choose a proper sub-optimal window size to still obtain a reasonable convergence rate for the test statistic. The second
and more severe problem is that the ratio statistic proposed in \cite{JP13} heavily relies on the scaling property of a Brownian motion. This scaling property is not shared by an i.i.d. noise process introduced in \eqref{noise}. Thus, a much deeper probabilistic analysis of the main statistic is required
to come up with a valid testing procedure.

The paper is organized as follows. Section \ref{sec2} gives the probabilistic description of the model, presents the main assumptions 
and defines the testing hypotheses. The background on matrix perturbation and pre-averaging method is demonstrated in Section \ref{sec3}.
Section \ref{sec4} presents the main results of the paper. 
Section \ref{sec6} is concerned with a simulation study.
All proofs are collected in Section \ref{sec5}.

\section{The setting and main assumptions}
\label{sec2}
\setcounter{equation}{0}
\renewcommand{\theequation}{\thesection.\arabic{equation}}
We start with a filtered probability space $(\Omega, \mathcal F, (\mathcal F_t)_{t\in [0,T]}, \mathbb P)$, on which all stochastic processes
are defined. As indicated at \eqref{diffusion} we observe the $d$-dimensional process $Y=X+\varepsilon$ at time points $i\Delta_n$, $i=0,1,\ldots, [T/\Delta_n]$.
The process $X$ is given via
\begin{align} \label{X}
X_t= X_0+ \int_0^t b_s ds + \int_0^t \sigma_s dW_s,
\end{align}
where $(b_t)_{t\in [0,T]}$ is a $d$-dimensional drift process, $(\sigma_t)_{t\in [0,T]}$ is a $\R^{d \times q}$-valued volatility process and $W$ denotes
a $q$-dimensional Brownian motion. We introduce the notation 
\begin{align*}
c_t=\sigma_t\sigma_t^{\star},\qquad r_t=\rank(c_t),\qquad
R_t=\sup_{s\in[0,t)}\,r_s.
\end{align*}
We need more structural assumptions on the processes $b$ and $\sigma$. 
\vsq

\nib Assumption (A): \rm The processes $b$ and $\sigma$ have the form
\begin{align}
\sigma _t&=\sigma_0+ \int_0^t a_s ds + \int_0^t v_s dW_s, \nonumber\\[1.5 ex]
\label{bsigma} b_t&=b_0+ \int_0^t a'_s ds + \int_0^t v'_s dW_s, \\[1.5 ex]
v_t &=v_0+ \int_0^t a''_s ds + \int_0^t v''_s dW_s, \nonumber
\end{align}
where  $b_t$ and $a'_t$ are
$\R^d$-valued, $\sigma_t$, $a_t$ and $v'_t$ are
$\R^{d\times q}$-valued, $v_t$ and $a''_t$ are $\R^{d\times q\times q}$-valued,
and $v''_t$ is $\R^{d\times q\times q\times q}$-valued, all those processes
being adapted. Finally, the processes
$a_t,v'_t,v''_t$ are c\`adl\`ag and the processes $a'_t,a''_t$
are locally bounded.\qed
\vsq

Notice that (A) is exactly the same assumption, which has been imposed in \cite{JP13}. We remark that, by enlarging the dimension $q$
of the Brownian motion $W$ if necessary, we may assume without loss of generality that all processes $X,b,\sigma,v$ are driven by the same Brownian motion.
In the framework of a stochastic differential equation, i.e. when $b_t=h_1(X_t)$ and $\sigma_t=h_2(X_t)$, assumption (A) is automatically satisfied whenever
$h_1\in C^2(\R)$ and $h_2\in C^4(\R)$ (due to It\^o's formula). We also remark that assumption (A) is rather unusual in the literature. Indeed, for
classical high frequency statistics, such as e.g. power variations (cf. \cite{BGJPS06}), only the first line of \eqref{bsigma} is required. However, 
when $R_T<d$ our test statistic, which will be introduced in Section \ref{sec4}, turns out to be degenerate and, in contrast to classical cases, 
we require a higher order stochastic expansion of the increments of $X$. This explains the role of the second and third line of \eqref{bsigma}. Finally, we specify our assumptions on the noise process $\ep$ introduced at \eqref{noise}.
\vsq

\nib Assumption (E): \rm The i.i.d. process $(\ep)_{t\in[0,T]}$ is $(\mathcal F_t)$-adapted and independent of $a$, $a'$, $a''$, $v'$, $v''$, $W$, hence also independent
of $b$, $\sigma$ and $X$. Furthermore, it is Gaussian, meaning that $\ep_t \sim \mathcal N_d(0, \Sigma)$ and $\E[ \ep_s \ep_t^\star]=0$ for all $s, t \in [0,T]$ with $s\neq t$.
\qed
\vsq 
\begin{rem}\rm
Theoretically, we could discuss a more general structure of the noise. In particular, we could give up the assumption of the Gaussianity. What we really require is the mutual independence of the noise at different times as well as the existence of the moments up to a certain order. Also  the independence assumption 
between the noise $\ep$ and the semimartingale $X$ could be generalized; see e.g. \cite{JLMPV09, PV09b} for an exposition of the details.
\qed
\end{rem}
Now, for any $r\in \{0,1,\ldots, d\}$, we introduce the following subsets of $\Omega$:
\begin{align} \label{omegar}
\Omega_T^r := \{\omega \in \Omega:~R_T(\omega)=r\}, \qquad \Omega_T^{\leq r} := \{\omega \in \Omega:~R_T(\omega)\leq r\}.
\end{align} 
Notice that the sets $\Omega_T^r$ and $\Omega_T^{\leq r}$ are indeed $\mathcal F_T$-measurable. This can be justified as follows. 
The rank $r_t$ is the biggest integer $r\leq d$ such that the sum of the
determinants of the matrices $(c_t^{ij})_{i,j\in J}$, where $J$ runs through
all subsets of $\{1,\ldots,d\}$ with $r$ points, is positive; see e.g. \cite[Lemma 3]{JLT08}. Since the mapping $t\mapsto c_t$ is
continuous by assumption (A), this implies that for any $r$ the random set
$\{t: r_t(\om)> r\}$ is open in $[0,T)$, so the mapping $t\mapsto r_t$ is lower
semi-continuous. The very same argument proves that the random set
\[
\{t\in [0,T):~R_T(\omega)=r_t(\omega)\} 
\]    
is non-empty and open for each $\omega \in \Omega$. Hence, this set has a positive Lebesgue measure, which helps to statistically 
identify the maximal rank $R_T$ (in contrast to lower ranks $r_t<R_T$, which might be attained at a single point on the interval $[0,T]$).

The following discussion is devoted to testing the null hypothesis $H_0:~R_T=r$ against the alternative $H_1:~R_T \not =r$ (or $H_0:~R_T\leq r$
against $H_1:~R_T >r$). Notice that this a \textit{pathwise hypothesis}, since we test whether a given path $\omega$ belongs to $\Omega_T^r$
(or $\Omega_T^{\leq r}$)
or not. It is in general impossible to know whether this hypothesis holds for another path $\omega'\in \Omega$.

\section{Matrix perturbation and pre-averaging approach}
\label{sec3}
\setcounter{equation}{0}
\renewcommand{\theequation}{\thesection.\arabic{equation}}

\subsection{Matrix perturbation method} \label{sec3.1}
The matrix perturbation method is a numerical approach to the computation of the rank of a given
matrix. It has been introduced in \cite{JP13} in the context of rank testing. 
To explain the main idea of our method, we need to introduce
some notation. Recall that $d$
and $q$ are the dimensions of $X$ and $W$, respectively. Let $\m$
denote the set of all $d\times d$ matrices and $\m_r$,  
$r\in\{0,\ldots,d\}$, the set of all matrices in $\m$ with rank $r$. Furthermore,
let $\m'$ be the set of all $d\times q$ matrices. For any matrix $A$
we denote by $A_i$ the $i$th column of $A$; for any vectors
$x_1,\ldots,x_d$  in $\R^d$, we write mat$(x_1,\cdots,x_d)$ for the
matrix in $\m$ whose $i$th column is the column vector $x_i$. For
$r\in\{0,\ldots,d\}$ and $A,B\in\m$ we define the quantity
\begin{align*} 
\m^r_{A,B} :=~ \big\{ G\in \m: \, G_i=A_i ~\text{or}~ G_i=B_i
~\text{with}~\#\{i:\,G_i=A_i\}=r\big\}.
\end{align*}
In other words,  $\m^r_{A,B}$ is the set
of all matrices $G\in\m$ with $r$ columns equal to those of $A$
and the remaining $d-r$ ones equal to those of $B$ (all of them being at their original places).
We also  define the number
\begin{align} \label{gamma}
\ga_r(A,B):=\sum_{G\in\m^r_{A,B}}\det(G).
\end{align}
We demonstrate the main ideas of the matrix perturbation approach for a deterministic problem first. Let
$A\in \m$ be an unknown matrix with rank $r$. Assume that, although
$A$ is unknown, we have a way of computing $\det (A+\lambda B)$ for all $\lambda>0$ and
some given matrix $B\in\m_d$. The multi-linearity property of the
determinant implies the following asymptotic expansion
\begin{align} \label{detexp}
\det(A+\lambda B) = \lambda^{d-r} \ga_r(A,B) + O(\lambda^{d-r+1}) \qquad \text{as} \quad \lambda\downarrow 0.
\end{align}
This expansion is the key to identification of the unknown rank $r$.
Indeed,  when $\ga_r(A,B) \not = 0$ we deduce that 
\begin{align} \label{detconv}
\frac{\det (A+2\lambda B)}{\det (A+\lambda B)} \rightarrow 2^{d-r} \qquad \text{as}
\quad \lambda\downarrow 0.
\end{align}
However, it is
impossible to choose a matrix $B\in \m$ which guarantees $\ga_r(A,B) \not = 0$
for all $A\in\m_r$. To solve this problem, we can use a \textit{random} perturbation.
As it has been shown in \cite{JP13}, for any $A\in\m_r$ we have $\ga_r(A,B) \not = 0$ almost surely 
when $B$ is the random matrix whose entries are independent  standard normal (in fact, the random variable $\ga_r(A,B)$
has a Lebesgue density). This is intuitively clear, because the multivariate standard normal distribution does not prefer directions.
It is exactly  this idea which will be the core of our testing procedure.

\subsection{Pre-averaging approach}  \label{sec3.2}
In this subsection, we briefly introduce the pre-averaging method; we refer to e.g. \cite{JLMPV09,PV09} for a more detailed exposition. 

Let $g:[0,1]\rightarrow \R$ be a weight function with $g(0)=g(1)=0$, which is continuous, piecewise $C^1$ with piecewise Lipschitz derivative $g'$ 
and $\int_0^1 g^2(x)dx>0$. A canonical choice of such a function is given by $g(x)= \min(x,1-x)$; see \cite{JLMPV09} for its interpretation.
Now, let $(k_n)_{n\ge1}$ be a sequence of positive integers representing the window size such that $k_n\rightarrow \infty$ and $u_n:=k_n \Delta_n\rightarrow 0$. For any stochastic process
$V$, we define the pre-averaged increments via
\begin{align} \label{prereturn}
\overline{V}_i^n := \sum_{j=1}^{k_n-1} g \left(\frac{j}{k_n} \right) \Delta_{i+j}^n V = - \sum_{j=0}^{k_n-1} 
\left( g \left(\frac{j+1}{k_n} \right) -g \left(\frac{j}{k_n} \right) \right) V_{(i+j)\Delta_n},  
\end{align} 
where $\Delta_i^n V:= V_{i\Delta_n} - V_{(i-1)\Delta_n}$. Roughly speaking, this local averaging procedure reduces the influence of the noise process
when we apply it to the noisy diffusion process $Y$ defined at \eqref{diffusion}. Indeed, we may show that 
\[
\overline{X}_i^n = O_{\mathbb P} (\sqrt{k_n \Delta_n}) \qquad \text{and} \qquad \overline{\varepsilon}_i^n = O_{\mathbb P} (\sqrt{1/k_n}),  
\]
where the first approximation is essentially justified by the independence of the increments of $W$ and the first identity of \eqref{prereturn},
and the second approximation follows from the i.i.d. structure of the noise process and the second identity of \eqref{prereturn}. We clearly
see that a large $k_n$ increases the influence of the diffusion part $X$ and diminishes the influence of the noise part $\varepsilon$. However, 
in standard statistical problems, e.g. estimation of quadratic variation, the optimal rate of convergence is obtained when the contributions 
of both terms  are balanced. This results in the choice of the window size $k_n$ with $k_n \sqrt{\Delta_n} = \theta + o(\Delta_n^{1/4})$, where $\theta \in(0,\infty)$. 
With this window size we deduce for instance that
\[
\sqrt{\Delta_n} \sum_{i=0}^{[T/\Delta_n] - k_n+1} (\overline{Y}_i^n) (\overline{Y}_i^n)^{\star} \toop \theta \psi_2 \int_0^T c_t dt +
\theta^{-1} \psi_1 T\Sigma , 
\]
where the constants $\psi_1$ and $\psi_2$ are defined by
\begin{align*}
\psi_1:= \int_0^1 (g'(x))^2 dx, \qquad \psi_2:= \int_0^1 g^2(x) dx, 
\end{align*} 
(cf. \cite{JLMPV09}). The bias can be corrected via
\[
C_t^n = \frac{\sqrt{\Delta_n}}{\theta \psi_2} \sum_{i=0}^{[T/\Delta_n] - k_n+1} (\overline{Y}_i^n) (\overline{Y}_i^n)^{\star} 
- \frac{\psi_1 \Delta_n}{2\theta^2 \psi_2} \sum_{i=0}^{[T/\Delta_n]} (\Delta_i^n Y) (\Delta_i^n Y)^{\star} \toop \int_0^T c_t dt,
\]
and the statistic $C_t^n$ becomes a consistent estimator of the quadratic covariation of $X$ with convergence rate $\Delta_n^{-1/4}$.
This rate is known to be optimal.

As explained in the introduction, the optimal choice of the window size $k_n$ as introduced above would not lead to a feasible testing
procedure for the maximal rank $R_T$. Due to the complex structure of the test statistic, which will be introduced in Section \ref{sec4},
there is no de-biasing procedure as above (when there are no further restrictions on the rank of the covariance matrix $\Sigma$). For this
reason we introduce the following window size $k_n$:
\begin{align} \label{kn}
k_n \Delta_n^{2/3} = \theta + o(\Delta_n^{1/6}), \quad \theta \in (0,\infty), \qquad u_n: =k_n\Delta_n.
\end{align}  
Within the framework of our test statistic, this choice of $k_n$ leads to an optimal rate of convergence, which becomes $\Delta_n^{-1/6}$
(although better rates of convergence are theoretically possible when using alternative test statistics). 
We show the intuition behind this choice in the next section. We remark that
an easier choice of the window
size would be $k_n= O(\Delta_n^{-3/4})$, which would completely eliminate the influence of the noise process on the central limit theorem.
However, this would lead to a slower rate of convergence $\Delta_n^{-1/8}$. For this reason we dispense with the exact exposition of this case.

\section{Main results}
\label{sec4}
\setcounter{equation}{0}
\renewcommand{\theequation}{\thesection.\arabic{equation}}

\subsection{Test statistic} \label{sec4.1}
In this subsection, we introduce a random perturbation of the original data and define the main statistics. 
Following the basic ideas of \cite{JP13} and the motivation of Subsection \ref{sec3.1}, we define a  $d$-dimensional `perturbation' process $X'$ by
\begin{align*}
	X'_t= \widetilde{\sigma} W'_t,
\end{align*}
where $\widetilde{\sigma} $ is a positive definite deterministic $d \times d$ matrix and $W'$ is a $d$-dimensional Brownian motion. Without loss of generality, we may assume that $W'$ is also defined on the filtered probability space
$(\Omega, \mathcal F, (\mathcal F_t)_{t\in [0,T]}, \mathbb P)$. Let $\mathcal G \subset \mathcal F$ be the 
sub-$\sigma$-algebra, which is generated by all processes appearing in (A) and by the noise process $\ep$. We assume that $W'$ is independent of $\mathcal G$.
 Now,  we use $X'$ to define the perturbed process
\begin{align}\label{Z}
	Z_t^{n,\kappa}:= Y_t + \sqrt{\kappa u_n} X'_t,
\end{align}
where $\kappa=1,2$ and the sequence $u_n$ is defined at \eqref{kn}. In some sense, the perturbation process $X'_t$ plays the role of the random perturbation
matrix $B$ introduced after \eqref{detconv}. As we will see below, our two main statistics will be constructed at two different frequencies $\Delta_n$
and $2\Delta_n$, which will be indicated by the constant $\kappa=1,2$. 

Recall the definition of the pre-averaged quantity $\overline{V}_i^n$ introduced in \eqref{prereturn} for a stochastic process $V$. We sometimes write $\overline{V}(g)_i^n$
instead of $\overline{V}_i^n$ if we want to stress the dependency of the term $\overline{V}_i^n$ on the weight function $g$. Furthermore, we use
the notation $\overline{V}(g)_i^{n,\kappa}$ to indicate that the quantity $\overline{V}(g)_i^n$ is built using frequency $\kappa \Delta_n$ with $\kappa=1,2$, i.e.
\begin{equation}\label{eqn: V(g)}
\overline{V}(g)_i^{n,\kappa}=  \sum_{j=1}^{k_n-1} g \left(\frac{j}{k_n} \right) (V_{(i+\kappa j)\Delta_n} - V_{(i+\kappa (j-1))\Delta_n}). 
\end{equation}
If $V = Z^{n,\kappa}$ defined at \eqref{Z}, we will slightly abuse the notation introduced in \eqref{eqn: V(g)} and use the convention $\overline Z(g)_i^{n,\kappa} := \overline{Z^{n,\kappa}}(g)_i^{n,\kappa}$.
Now, we define our main test statistics via
\begin{align} \label{main-statistic}
	S(g)_T^{n,\kappa}
	=3d u_n \sum_{i=0}^{[T/3d u_n]-1}
	f\left(\left({\overline Z(g)^{n,\kappa}_{((3i+\kappa-1)d+\kappa (j-1))k_n}}\big/{\sqrt{\kappa u_n}}\right)_{j=1,\ldots, d}\right),
\end{align}
for $\kappa=1,2$ with the test function $f$ on $(\R^d)^d$ given as
\begin{align}\label{f}
	f(x_1, \ldots, x_d):= \text{det}(\text{mat}(x_1,\ldots,x_d))^2.
\end{align}
Note that the summands in \eqref{main-statistic} use non-overlapping  increments of the process $Z^{n,\kappa}$, and also the statistics $Z^{n,1}$
and $Z^{n,2}$ are based on distinct increments.

\begin{rem} \label{rem1} \rm
The statistic $S(g)_T^{n,\kappa}$ is similar in spirit to the one introduced in \cite{JP13}, where a $d$-dimensional
continuous It\^o semimartingale without noise has been considered. Therein, the statistics $S_T^{n,\kappa}$
defined in \cite[Equation (2.13)]{JP13}, which use the raw increments instead of pre-averaged ones, satisfy the following law of large numbers
\[
\frac{S_T^{n,2}}{S_T^{n,1}} \toop 2^{d-R_T}.
\] 
This should be compared with the motivation described at \eqref{detconv}. The latter convergence asymptotically identifies the maximal rank
$R_T$. The crucial difference to our framework is that this convergence is no longer valid when we use the statistics $S(g)_T^{n,\kappa}$
introduced in \eqref{main-statistic}. It relies on the fact that the noise process $\varepsilon$ does not have the scaling property 
of the driving Brownian motion $W$. To overcome this issue, we will not only use different frequencies $\Delta_n$ and $2\Delta_n$, but also two different
weight functions $g$ and $h$, which are connected through certain identities. For this purpose a very thorough analysis of the asymptotic behaviour 
of $S(g)_T^{n,\kappa}$ is required. \qed  
\end{rem}

\begin{rem} \label{rem2} \rm
Let us explain the choice of the window size $k_n$ introduced at \eqref{kn} and the perturbation rate $\sqrt{\kappa u_n}$. Under assumptions (A) and (E)
we will prove the following asymptotic decomposition for $i=0,\ldots, [T/3d u_n]-1$
\begin{align}
\label{stoch decomp}
&\frac{1}{\sqrt{\kappa u_n}} \text{mat} \left(\overline Z(g)^{n,\kappa}_{(3i + \kappa - 1)dk_n}, \cdots, \overline Z(g)^{n,\kappa}_{((3i + \kappa - 1)d+\kappa(d-1))k_n}\right)\\[1.5 ex]
& = A(g)_{i}^{n,\kappa}
+ \sqrt{\kappa u_n} \left(B(1,g)_{i}^{n,\kappa} + B(2,g)_{i}^{n,\kappa} + B(3,g)_{i}^{n,\kappa} \right) + \kappa u_n C(g)_{i}^{n,\kappa} + \kappa u_n D(g)_{i}^{n,\kappa},  \nonumber
\end{align}
where the $\R^{d\times d}$-valued sequences $ A(g)_{i}^{n,\kappa}, C(g)_{i}^{n,\kappa}, D(g)_{i}^{n,\kappa}$ and $B(g)_i^{n,\kappa} := B(1,g)_{i}^{n,\kappa} + B(2,g)_{i}^{n,\kappa} + B(3,g)_{i}^{n,\kappa}$ are tight. The matrix $A(g)_{i}^{n,\kappa}$,
which is the dominating term in the expansion, is defined by
\begin{align*}
A(g)_{i}^{n,\kappa}= 
\frac{\sigma_{(3i+\kappa-1)du_n}}{\sqrt{\kappa u_n}} \text{mat} \left(\overline W(g)^{n,\kappa}_{(3i + \kappa - 1)dk_n}, \cdots, \overline W(g)^{n,\kappa}_{((3i + \kappa - 1)d+\kappa(d-1))k_n}\right),
\end{align*}
while $B(1,g)_{i}^{n,\kappa}$ depends on $b,v$ introduced in \eqref{bsigma}, $B(2,g)_{i}^{n,\kappa}$ comes solely from the perturbation $X'$
and $B(3,g)_{i}^{n,\kappa}$ is associated with the noise process $\varepsilon$ (the third order term $C(g)_{i}^{n,\kappa}$ is connected to $a,v',v''$
and the term $D(g)_{i}^{n,\kappa}$ depends on $a, a', a'', v', v''$,  defined in \eqref{bsigma}). Since $\text{det}(A(g)_{i}^{n,\kappa})=0$ whenever $R_T<d$, our statistic $S(g)_T^{n,\kappa}$ is degenerate in the sense
that the second order term enters the law of large numbers. At this stage, we realize that the choice of the window size $k_n=O(\Delta_n^{-2/3})$
and the perturbation rate $\sqrt{\kappa u_n}$ creates a balance between the second order term $B(g)_i^{n,\kappa}$ in the stochastic expansion coming from the diffusion process,
the noise process $\varepsilon$ and the perturbation process $X'$. The classical choice $k_n=O(\Delta_n^{-1/2})$ would make the noise part one of the dominating
terms, but in this case, the estimation of the maximal rank $R_T$ would be virtually impossible since we impose no assumptions on the covariance matrix $\Sigma$
of the noise. On the other hand, when $k_n=O(\Delta_n^{-3/4})$ the noise part would enter the third order term and thus would not influence the limit theory.
Although the asymptotic results become much easier in the latter case, the convergence rate gets rather low ($\Delta_n^{-1/8}$). Hence, within the framework
of our test statistic, the choice $k_n=O(\Delta_n^{-2/3})$ meets the balance between feasibility of the testing procedure and the optimal rate of convergence.

Clearly, $B(g)_{i}^{n,\kappa}$ plays the role of the perturbation matrix $B$ defined in section \ref{sec3.1}
while $\lambda=\sqrt{\kappa u_n}$. Since it is impossible to guarantee that the matrices $B(1,g)_{i}^{n,\kappa}$ and $B(3,g)_{i}^{n,\kappa}$ have full rank, we require the 
presence of the matrix $B(2,g)_{i}^{n,\kappa}$ to insure almost sure invertibility of the sum. Thus, the perturbation process $X'$ plays the role of regularization.
\qed 
\end{rem}

\subsection{Notation}\label{sec: Notation}
In order to state the limit theory for the statistics $S(g)_T^{n,\kappa}$, we need to introduce some notation.
For any weight function $g$ we define the quantities
\begin{align}\label{psi}
&\psi_1(g)= \int_0^1 (g'(x))^2 dx,  &&\psi_2(g)= \int_0^1 g^2(x) dx, \\
&\psi_3(g)= \int_0^1 g(x) dx,  &&\psi_4(g)= \int_0^1 xg^2(x) dx. \nonumber
\end{align}
For $r\in\{0,1,\ldots, d\}$, we define the function $F_r$ on $(\R^{2d})^d$ by
\begin{align}\label{F_r}
F_{r}(v_1,\ldots,v_d)=
	\gamma_r\left(\text{mat}(x_1,\ldots,x_d),\text{mat}(y_1,\ldots,y_d)
\right)^2,
\qquad  v_j=\begin{pmatrix} x_j \\ y_j \end{pmatrix}
\in \R^{2d},
\end{align}
where $\gamma_r$ was introduced at \eqref{gamma}.
Let $\overline W$ and $\overline W'$ be Brownian motions of dimension $q$ and $d$, respectively, and let $\overline\Theta=(\overline\Theta_i)_{i\ge1}$ be a i.i.d. sequences of $d$-dimensional standard normal random variables. $\overline W$, $\overline W'$, and $\overline\Theta$ are defined on some filtered probability space $(\overline\Omega, \overline{\mathcal{F}}, (\overline{\mathcal{F}}_t)_{t\ge0}, \overline{\mathbb P})$ and are assumed to be independent. Let $\mathcal M^{\ge0}$ be the space of all symmetric positive-semidefinite matrices $\varphi\in\mathcal M$. We introduce the space
$\mathcal U=\mathcal M' \times \mathcal M \times \R^{dq^2} \times \R^d \times \mathcal M^{\ge0}$, and let $\underline u=(\alpha,\beta,\gamma, a, \varphi)\in \mathcal U$. By $\varphi^{1/2}\in\mathcal M$ we denote the matrix root of $\varphi$.

Now, for $\kappa=1,2$, we define the $2d$-dimensional variables (explicitly writing the components with $l\in\{1,\ldots, d\}$)
\begin{align}
\label{Psi1}
	&\Psi(\underline u, g, \kappa)^l_j
	=\frac{1}{\sqrt \kappa} \sum_{m=1}^q \alpha^{lm}
	\int_{\kappa(j-1)}^{ \kappa j}  g(s/\kappa -  (j-1)) \, d\overline W^m_{s},\\[1.5ex]
\label{Psi2}
	&\Psi(\underline u, g, \kappa)^{d+l}_j
	= \frac{1}{\kappa} a^l \int_{\kappa(j-1)}^{ \kappa j}  g(s/\kappa-  (j-1)) \, ds  \\[1.5ex]
	&\qquad+ \frac{1}{\kappa} \sum_{m,k=1}^q \gamma^{lkm} \int_{\kappa(j-1)}^{ \kappa j} 
	g(s/\kappa- (j-1)) \,\overline W^k_s d\overline W^m_s\nonumber \\[1.5ex]
	&\qquad +\frac{1}{\sqrt \kappa} \sum_{m=1}^d \beta^{lm}
	\int_{\kappa(j-1)}^{ \kappa j}  g(s/\kappa- (j-1)) \, d\overline W'^m_{s} \nonumber \\[1.5ex]
	&\qquad+\frac{1}{ \kappa} \left( \frac{\psi_1(g)}{\theta^3}\right)^{1/2} \sum_{m=1}^d \left(\varphi^{1/2}\right)^{lm}\,\overline\Theta_{\kappa j}^m. \nonumber 
\end{align}
Some explanations are in order to understand these definitions. 

\begin{rem} \label{rem3} \rm
To get an intuition for the notation we remark that the components of $\underline u\in\mathcal U$ account for the processes in assumption (A) that will appear in the limit. This means that $\alpha$ is related to $\sigma_t$, $\beta$ to $\widetilde{\sigma}$, $\gamma$ to $v_t$ and $a$ to $b_t$. Finally, $\varphi$  accounts for the covariance structure of the noise and is associated with $\Sigma$.
As motivated above we use different rates in our procedure. Therefore, we also have to define the limit for the two cases $\kappa=1,2$. \qed
\end{rem}
\begin{rem}  \rm
Note that the random-vectors $\Psi(\uu,g,\kappa)_i$ and $\Psi(\uu,g,\kappa)_j$ are uncorrelated whenever $i\neq j$.
\qed
\end{rem}
Using the notation at \eqref{F_r}, we define for a weight function $g$, $\underline u 
=(\alpha,\beta,\gamma, a, \varphi)\in \mathcal U$ and $\kappa=1,2$ the real-valued random variables
\begin{align*} 
\overline F_{r}(\underline u, g, \kappa)
=F_{r}\left(\Psi(\underline u, g,\kappa)_1,\ldots,
\Psi(\underline u, g,\kappa)_d\right),
\end{align*}
and set
\begin{align}
\label{Gammar}
\Gamma_{r}(\underline u, g, \kappa)
	&=\overline\E\left[\overline F_{r}(\underline u,  g, \kappa)\right], \\[1.5 ex]
\Gamma'_{r}(\underline u, g, \kappa)
	&=\overline\E\left[\overline F_{r}(\underline u, g, \kappa)^2\right]
	-\Gamma_{r}(\underline u,  g, \kappa)^2 \nonumber. 
\end{align}
\begin{rem}  \rm
Under the special assumption that $\varphi=0$ (which corresponds to the situation without noise), the sequences $(\Psi(\uu,g,1)_j)_{j\ge1}$ and $(\Psi(\uu,g,2)_j)_{j\ge1}$ have the same global law which implies also that $\Gamma_{r}(\underline u, g, 1) = \Gamma_{r}(\underline u, g, 2)$ and $\Gamma'_{r}(\underline u, g, 1) = \Gamma'_{r}(\underline u, g, 2)$. This is not the case when $\varphi \neq0$. Proposition \ref{prop1} will demonstrate under which conditions one can find another weight function $h$ such that $\Gamma_{r}(\underline u, g, 1) = \Gamma_{r}(\underline u, h, 2)$ and $\Gamma'_{r}(\underline u, g, 1) = \Gamma'_{r}(\underline u, h, 2)$ even in the general situation that $\varphi \neq0$. \qed
\end{rem}

\begin{rem} \label{rem4} \rm
We have introduced the random variables $\Psi(\underline u, g, \kappa)_j$ only for weight functions, implying that $g$ is continuous and piecewise $C^1$ with a piecewise Lipschitz derivative $g'$. As a matter of fact, we will often work with a discretized version $g^n$ of $g$ defined as
\begin{align}\label{g^n}
g^n(s) := \sum_{i=1}^{k_n-1} g\left(\frac{i}{k_n} \right) \mathbf{1}_{\left(\frac{i-1}{k_n}, \frac{i}{k_n}\right]}(s).
\end{align}
Note that $g^n(0) = g^n(1) = 0$ and that $g^n$ converges to $g$ uniformly on $[0,1]$. By definition, $g^n$ fails to be a weight function as it is not continuous. Nevertheless, the integrals $\int_0^1g^n(s)ds$, $\int_0^1g^n(s)d\overline W^m_s$ and $\int_0^1g^n(s)\overline W^k_s d\overline W^m_s$ still make sense. This corresponds to the fact that $\psi_l(g^n)$ introduced at \eqref{psi} is well-defined for $l=2,3,4$. Moreover, we have by a Riemann approximation argument that
\[
\psi_l(g^n) = \psi_l(g) + O(k_n^{-1}), \quad l=2,3,4.
\]
For $\psi_1(g^n)$, we must approximate the derivative and set
\begin{align}\label{psi_1(gn)}
\psi_1(g^n) := \frac{1}{k_n} \sum_{i=0}^{k_n-1} \left(\frac{g\left(\frac{i+1}{k_n} \right) - g\left(\frac{i}{k_n} \right)}{1/k_n}\right)^2 
= \psi_1(g) + O(k_n^{-1}),
\end{align}
where the second identity follows again by a Riemann approximation argument.
With this convention, we can extend the notation and write $\Psi(\underline u, g^n, \kappa)_j$, $\overline F_{r}(\underline u,  g^n, \kappa)$, $\Gamma_{r}(\underline u, g^n, \kappa)$ and $\Gamma'_{r}(\underline u, g^n, \kappa)$, respectively. \qed
\end{rem}

\subsection{Law of large numbers}
In this subsection, we present the law of large numbers for the statistic $S(g)_T^{n,\kappa}$. The quantity $\Gamma_{r}(\underline u, g, \kappa)$
defined at \eqref{Gammar} will essentially determine the limit. First, we demonstrate how the terms $\Gamma_{r}(\underline u, g, \kappa)$ and $\Gamma'_{r}(\underline u, g, \kappa)$ 
depend on the rank of the argument $\alpha$. The following lemma has been shown in \cite[Lemma 3.1]{JP13}.
\begin{lem} \label{lemma:Gamma}
Let $\underline u=(\alpha,\beta,\gamma,a, \varphi) \in \mathcal U$ with $\beta \in \mathcal{M}_d$ and $g$ be a weight function. 
	Then, if $r\in \{0,\ldots,d\}$ and $\kappa=1,2$, we deduce that
	\begin{align}
	\label{eqn: positive}
		&\rank(\alpha) =r  \ \Longrightarrow \ \Gamma_r(\underline u,  g,\kappa)>0, \ \Gamma'_{r}(\underline u, g,\kappa) >0, \\
	\label{eqn: positive2}
		& \rank(\alpha) < r \ \Longrightarrow \
		\Gamma_{r}(\underline u, g,\kappa)= \Gamma'_{r}(\underline u, g,\kappa)=0.
	\end{align}
\end{lem}
The law of large numbers is as follows.

\begin{theo} \label{th1}
Assume that conditions (A) and (E) hold. Let $r\in \{0,\ldots,d\}$ and $g$ be a weight function. Then, on $\Omega_T^r$ and for $\kappa=1,2$, we obtain the convergence
\begin{align} \label{lln}
(\kappa u_n)^{r-d} S(g)_T^{n,\kappa} \toop S(r,g)_T^{\kappa}:= \int_0^T \Gamma_r(\sigma_s, \widetilde{\sigma}, v_s,b_s, \Sigma,  g,\kappa) ds>0. 
\end{align}  
\end{theo}

In view of Remark \ref{rem1}, Theorem \ref{th1} is not directly applicable since the limit $S(r,g)_T^{\kappa}$ crucially depends on $\kappa$, meaning that generally $S(r,g)_T^{1} \neq S(r,g)_T^{2}$.
In particular, the ratio statistics $S(g)_T^{n,2}/S(g)_T^{n,1}$ does not contain any information about the unknown maximal rank $R_T$. To make
use of Theorem \ref{th1} we need a better understanding of the structure of the functional $\Gamma_r$. The following proposition is absolutely crucial
for our testing procedure.

\begin{prop} \label{prop1}
(i) Fix $r\in\{0,\ldots, d\}$, $\underline u \in \mathcal U$ and $\kappa=1,2$. Then there exist $C^\infty$-functions $\tau_{r,\underline u, \kappa}, \tau'_{r,\underline u, \kappa} \colon \R^4 \to \R$ such that 
\begin{align} \label{psidep}
\Gamma_r(\underline u,  g,\kappa) = \tau_{r,\underline u, \kappa} (\psi_1(g), \ldots, \psi_4(g)), \quad 
\Gamma'_r(\underline u,  g,\kappa) = \tau'_{r,\underline u, \kappa} (\psi_1(g), \ldots, \psi_4(g))
\end{align}
for any weight function $g$.\\
(ii) Let $g$ and $h$ be weight functions such that $\psi_1(h) = 4 \psi_1(g)$ and $\psi_l(h) = \psi_l(g)$ for $l=2,3,4$.
Then, for any $r\in\{0,\ldots, d\}$ and any $\underline u\in \mathcal U$, we obtain that
\[
\Gamma_r(\underline u,  g,1) = \Gamma_r(\underline u,  h,2), \qquad \Gamma'_r(\underline u,  g,1) = \Gamma'_r(\underline u,  h,2).
\]
\end{prop} 
Proposition \ref{prop1}(i) says that the quantity $\Gamma_r(\underline u,  g,\kappa)$ does not depend on the entire function $g$, but only
on the quantities $\psi_l(g)$, $l=1, \ldots, 4$. But most importantly, Proposition \ref{prop1}(ii) and Theorem \ref{th1} imply the convergence
\begin{align} \label{finalconv}
\frac{S(h)_T^{n,2}}{S(g)_T^{n,1}} \toop 2^{d-r} \qquad \text{on } \Omega_T^r,
\end{align} 
whenever the pair of weight functions $g,h$ satisfies the conditions of Proposition \ref{prop1}(ii).
This opens the door to hypothesis testing. We now give an example of a pair of weight function $g,h$ which fulfills the conditions of Proposition \ref{prop1}(ii).

\begin{ex} \rm \label{ex:weight functions}
We define the two auxiliary weight functions $\widetilde g(x) := \max(0,\min(x,1-x))$ and $\widetilde h(x) := \max(0,\min(ax,b(1-x)))$ with $a=\frac{2}{2-\sqrt{3}}$ and $b=\frac{2}{2+\sqrt{3}}$. Then, a pair of weight functions satisfying the conditions of Proposition \ref{prop1}(ii) is given by $g_c(x):= \widetilde g(cx)$ and $h_c(x):= \widetilde h(cx -c + 1)$ where $c=\frac{8 + \sqrt{3}}{8}$ (see Figure \ref{fig:weight functions}). Indeed, for $c\ge1$, we obtain that
\begin{align*}
\psi_1(h_c) &= 4\psi_1(g_c) = 4c, \\
\psi_2(h_c) &= \psi_2(g_c) = \frac{1}{12c}, \\
\psi_3(h_c) &= \psi_3(g_c) = \frac{1}{4c}, \\
\psi_4(h_c) &= \frac{8c - 4-\sqrt{3}}{96c^2}, \quad \psi_4(g_c) = \frac{1}{24c^2}, 
\end{align*}
and for $c=\frac{8 + \sqrt{3}}{8}$, we have $\frac{8c - 4-\sqrt{3}}{96c^2} =  \frac{1}{24c^2} = \frac{536 - 128 \sqrt{3}}{11163}$.
\qed
\end{ex}

\begin{figure}
        \centering
        \begin{subfigure}[b]{0.5\textwidth}
                \includegraphics[width=\textwidth]{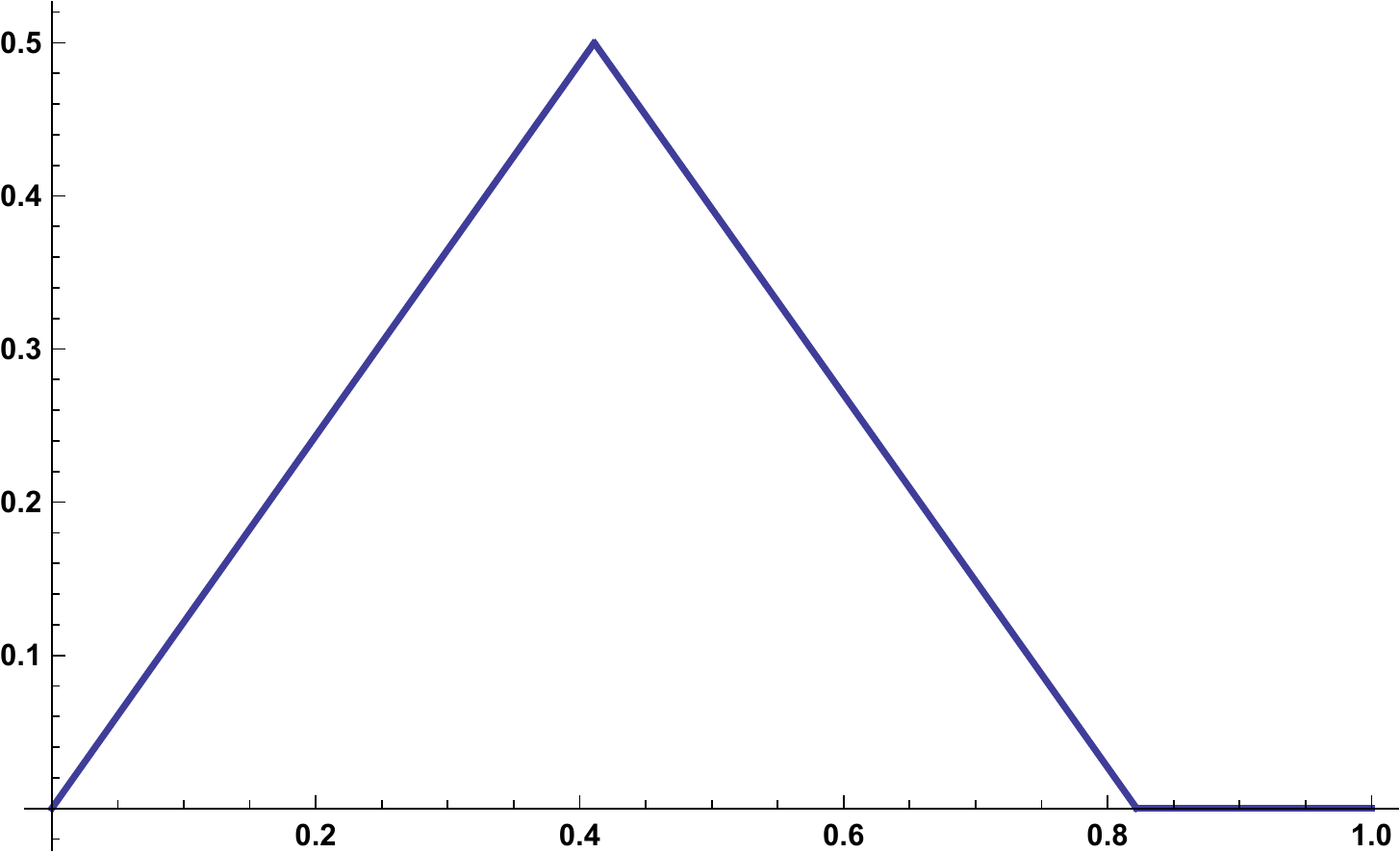}
                \caption{$g_c$ from Example \ref{ex:weight functions} with $c=\frac{8 + \sqrt{3}}{8}$.}
        \end{subfigure}%
        ~ 
        \begin{subfigure}[b]{0.5\textwidth}
                \includegraphics[width=\textwidth]{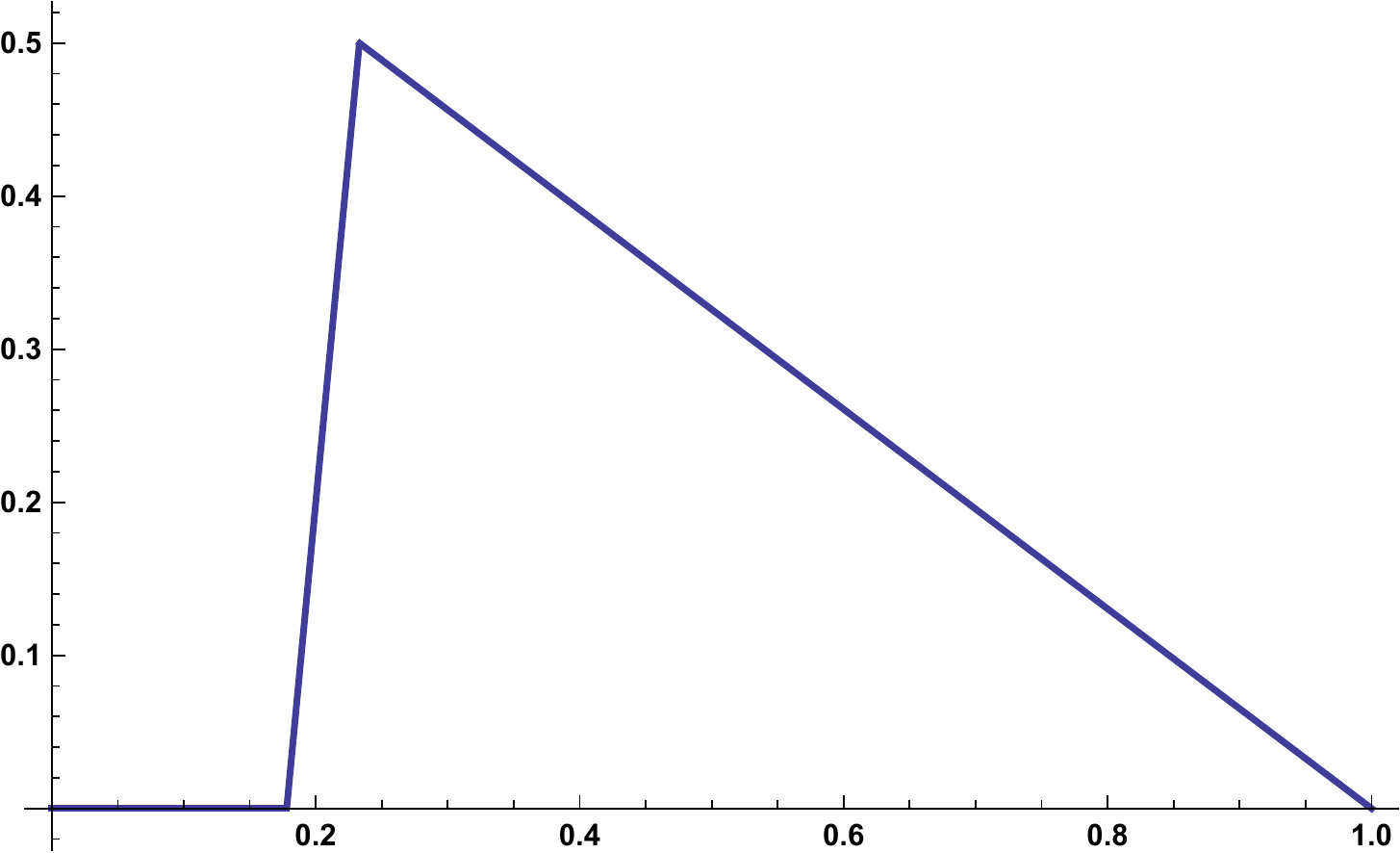}
                \caption{$h_c$ from Example \ref{ex:weight functions} with $c=\frac{8 + \sqrt{3}}{8}$.}
        \end{subfigure}
        \caption{A pair of weight functions satisfying the conditions of Proposition \ref{prop1}(ii).}\label{fig:weight functions}
\end{figure}

\begin{rem}\label{rem:weight functions} \rm
From a statistical point of view and regarding the definition of the pre-averaged increments in \eqref{prereturn}, we see that it is certainly not ideal to chose weight functions which are locally constant. Nevertheless, Example \ref{ex:weight functions} is an attempt to reduce the parts where the weight functions are constant while still sticking to a rather simple `triangular' form.
\hfill\qed
\end{rem}

\subsection{Central limit theorem and testing procedure}
In order to provide a formal testing procedure associated with the convergence in probability at \eqref{finalconv} we need to show a joint stable central limit theorem 
for the statistics $(S(g)_T^{n,1}, S(h)_T^{n,2} )$. We say that a sequence of random variables $H_n$ converges stably in law to
$H$ ($H_n \stab H$), where $H$ is defined on an extension $(\widetilde{\Omega}, \widetilde{\mathcal{F}}, \widetilde{\mathbb P})$ 
of the original probability space $(\Omega, \mathcal{F}, \mathbb P)$, if and only if
\begin{equation*} 
\lim_{n\rightarrow \infty} \E[\phi (H_n) Z] = \widetilde{\E}[\phi (H)Z]
\end{equation*} 
for any bounded and continuous function $\phi$ and any bounded $\mathcal F$-measurable
random variable $Z$.  We refer to \cite{AE78}, \cite{JS02} or \cite{R63} for a detailed study of stable convergence. 
Note that stable convergence is a stronger mode of convergence than weak convergence, but it is weaker than 
convergence in probability.

Now, let $g$ and $h$ be two weight functions satisfying the conditions of Proposition \ref{prop1}(ii). We define the statistic $U(r,g,h)_T^{n}=(U(r,g,h)_T^{n,1},
U(r,g,h)_T^{n,2})$ via 
\begin{equation}\label{U}
U(r,g,h)_T^{n} = \frac{1}{\sqrt{u_n}} \left(  u_n^{r-d} S(g)_T^{n,1} - S(r,g)_T^{1}, 
(2 u_n)^{r-d} S(h)_T^{n,2} - S(r,h)_T^{2}  \right). 
\end{equation}
The following theorem is one of the most important results of the paper.

\begin{theo} \label{th2}
Assume that conditions (A) and (E) are satisfied, the weight functions $g,h$ fulfill the assumptions of Proposition \ref{prop1}(ii) and $R_T(\omega)\leq r$
for some $r\in \{0, \ldots d\}$.
Then we obtain the stable convergence 
\begin{align} \label{clt}
U(r,g,h)_T^{n} \stab \mathcal{MN}(0,V(r,g,h)_T), 
\end{align}  
where 
\begin{equation}\label{V}
V(r,g,h)_T=  \diag\left(3d \int_0^T \Gamma'_r(\sigma_s, \widetilde{\sigma}, v_s,b_s, \Sigma ,  g,1) ds, 
3d \int_0^T \Gamma'_r(\sigma_s, \widetilde{\sigma}, v_s,b_s, \Sigma ,  h,2) ds \right)
\end{equation}
is a diagonal matrix. $\mathcal{MN}(0,V(r,g,h)_T)$ denotes the two dimensional mixed normal distribution with $\mathcal G$-conditional mean $0$ and 
$\mathcal G$-conditional covariance matrix $V(r,g,h)_T$.
\end{theo} 
Note that the rate of convergence $u_n^{-1/2}$ corresponds to $\Delta_n^{-1/6}$ for our choice of the window size $k_n$ at \eqref{kn}.
We remark that due to Proposition \ref{prop1}(ii), we know that $S(r,g)_T^{1} = S(r,h)_T^{2}$ such that the same centering term appears in both components on the right-hand side of \eqref{U}. Again thanks to Proposition \ref{prop1}(ii) we see that the two diagonal elements of $V(r,g,h)_T$ coincide. In order to obtain a feasible
version of the stable convergence in \eqref{clt}, we need to construct a consistent estimator of the $\mathcal G$-conditional covariance matrix $V(r,g,h)_T$. To this end, we define the following estimators for the `second moments':
\begin{align}\label{V(g)}
V(g,h)_T^{n,11}
	&=9d^2 u_n \sum_{i=0}^{[T/3d u_n]-1}
	f^2 \left(\left({\overline Z(g)^{n,1}_{(3id+ (j-1))k_n}}\big/{\sqrt{u_n}}\right)_{j=1,\ldots, d}\right), \\
\label{V(h)}
V(g,h)_T^{n,22}
	&=9d^2 u_n \sum_{i=0}^{[T/3d u_n]-1}
	f^2 \left(\left({\overline Z(h)^{n,2}_{(3id+ d+ 2(j-1))k_n}}\big/{\sqrt{2u_n}}\right)_{j=1,\ldots, d}\right), \\
\label{V(g,h)}
\begin{split}
V(g,h)_T^{n,12}
	&=9d^2 u_n \sum_{i=0}^{[T/3d u_n]-1}
	f  \left(\left({\overline Z(g)^{n,1}_{(3id+ (j-1))k_n}}\big/{\sqrt{u_n}}\right)_{j=1,\ldots, d}\right) \\
	&\hspace{10em} \times f \left(\left({\overline Z(h)^{n,2}_{(3id+ d+ 2(j-1))k_n}}\big/{\sqrt{2u_n}}\right)_{j=1,\ldots, d}\right),
\end{split}
\end{align} 
where $f$ is given at \eqref{f}.
Following the intuition from \eqref{finalconv} we define an estimator $\widehat{R}(g,h)_T^n$ via
\begin{align} \label{rt}
\widehat{R}(g,h)_T^n:= d - \frac{\log\left( S(h)_T^{n,2}\big/ S(g)_T^{n,1} \right)}{\log 2}. 
\end{align}  
Now, we obtain the following proposition.

\begin{prop} \label{prop2}
Assume that conditions (A) and (E) are satisfied and the weight functions $g,h$ fulfill the assumptions of Proposition \ref{prop1}(ii). \\
(i) Let $r\in \{0, \ldots, d\}$. Then, on $\Omega_T^{\le r}$: 
\begin{align}\label{V(g,h)_T^{n,11}}
(u_n^2)^{r-d}V(g,h)_T^{n,11} 
	&\toop 3d \int_0^T \Gamma'_r(\sigma_s, \widetilde{\sigma}, v_s,b_s, \Sigma,  g,1) + \Gamma^2_r(\sigma_s,
	\widetilde{\sigma}, v_s,b_s, \Sigma,  g,1)ds, \\
\label{V(g,h)_T^{n,22}}
(4 u_n^2)^{r-d} V(g,h)_T^{n,22} 
	&\toop 3d \int_0^T \Gamma'_r(\sigma_s, \widetilde{\sigma}, v_s,b_s, \Sigma,  h,2) + \Gamma^2_r(\sigma_s,
	\widetilde{\sigma}, v_s,b_s, \Sigma,  h,2)ds, \\
\label{V(g,h)_T^{n,12}}	
(2 u_n^2)^{r-d} V(g,h)_T^{n,12} 
	&\toop 3d \int_0^T  \Gamma_r(\sigma_s,\widetilde{\sigma}, v_s,b_s, \Sigma,  g,1)
	\Gamma_r(\sigma_s,\widetilde{\sigma}, v_s,b_s, \Sigma,  h,2)ds.
\end{align}
(ii) We have the (stable) central limit theorem 
\begin{align} \label{finalclt}
\frac{1}{\sqrt{u_n}}\frac{\widehat{R}(g,h)_T^n - R_T}{\sqrt{V(n,T,g,h)}} \stab  \Phi \sim \mathcal N(0,1),
\end{align}
where $\Phi$ is defined on an extension $(\widetilde{\Omega}, \widetilde{\mathcal{F}}, \widetilde{\mathbb P})$ 
of the original probability space $(\Omega, \mathcal{F}, \mathbb P)$ and is independent of the $\sigma$-algebra $\mathcal G$.
The random variable $V(n,T,g,h)$ is defined via
\begin{equation}\label{V(g)n}
V(n,T,g,h):= \frac{V(g,h)_T^{n,11} + 4^{\widehat{R}(g,h)_T^n -d}V(g,h)_T^{n,22} - 2^{1+ \widehat{R}(g,h)_T^n -d}V(g,h)_T^{n,12}  }{(S(g)_T^{n,1}\log 2)^2 }.
\end{equation}
\end{prop}
We remark that Proposition \ref{prop2}(ii) follows directly from Theorem \ref{th2}, Proposition \ref{prop2}(i) and the delta method for stable convergence. For this, it is essential to realize that, even though the estimator $V(n,T,g,h)$ for the conditional variance is not $\mathcal G$-measurable, it converges to a $\mathcal G$-measurable limit due to Proposition \ref{prop2}(i) and Theorem \ref{th1}.

Notice also that due to Proposition \ref{prop1}(ii) the right-hand side of \eqref{V(g,h)_T^{n,11}} and \eqref{V(g,h)_T^{n,22}} coincide and, moreover, that the right-hand side of \eqref{V(g,h)_T^{n,12}} can be written as 
\begin{equation*}
3d \int_0^T  \Gamma^2_r(\sigma_s,\widetilde{\sigma}, v_s,b_s, \Sigma,  g,1) ds.
\end{equation*}

\begin{rem}\label{remark4.14}\rm
Instead of using the estimators for the second moments given in \eqref{V(g)} till \eqref{V(g,h)}, we could also use a more direct approach and consider
\begin{align}\label{V'(g)}
\begin{split}
V'(g)_T^{n} &:= 3d^2 u_n \sum_{i=0}^{[T/2du_n]-1} \left\{ f\left( \left(\overline{Z}(g)^{n,1}_{(2id + j-1)k_n}\big/\sqrt{u_n}\right)_{j=1,\ldots, d} \right) \right.\\
&\hspace{12em}\left.- f\left( \left(\overline{Z}(g)^{n,1}_{(2id+d + j-1)k_n}\big/\sqrt{u_n}\right)_{j=1,\ldots, d} \right)\right\}^2.
\end{split}
\end{align}
Notice that similar to Proposition \ref{prop2}(i) we have that
\begin{align}\label{V'(g) convergence}
(u_n^2)^{r-d}V'(g)_T^{n} 
	\toop 3d \int_0^T \Gamma'_r(\sigma_s, \widetilde{\sigma}, v_s,b_s, \Sigma,  g,1) ds.
\end{align}
Then \eqref{finalclt} also holds upon replacing $V(n,T,g,h)$ defined at \eqref{V(g)n} by 
\begin{equation}\label{V'(g)n}
V'(n,T,g):= \frac{2V'(g)_T^{n}  }{(S(g)_T^{n,1}\log 2)^2 }.
\end{equation}
\qed
\end{rem}

The feasible central limit theorem at \eqref{finalclt} opens the door to hypothesis testing. Let us define the rejection regions via
\begin{align}
\label{test1} \mathcal{C}_{\alpha}^{n, =r}&:= \{\omega: ~|\widehat{R}(g,h)_T^n- r| > z_{1-\alpha/2} \sqrt{u_n V(n,T,g,h)}\}, \\[1.5 ex]
\label{test2} \mathcal{C}_{\alpha}^{n, \leq r}&:= \{\omega: ~\widehat{R}(g,h)_T^n- r > z_{1-\alpha} \sqrt{u_n V(n,T,g,h)}\},
\end{align}  
where $z_\alpha$ denotes the $\alpha$-quantile of the standard normal distribution. Obviously, the rejection region $\mathcal{C}_{\alpha}^{n, =r}$
corresponds to $H_0:~R_T=r$ vs. $H_1:~R_T \not=r$, while $\mathcal{C}_{\alpha}^{n, \leq r}$ corresponds to 
$H_0:~R_T\leq r$ vs. $H_1:~R_T >r$. The asymptotic level and consistency of the test are demonstrated in the following corollary.

\begin{cor}\label{corollary: testing}
Assume that conditions (A) and (E) are satisfied and the weight functions $g,h$ fulfill the assumptions of Proposition \ref{prop1}(ii). \\
(i) The test defined through \eqref{test1} has asymptotic level $\alpha$ in the sense that
\begin{equation}\label{test 1a}
A\subset \Omega_T^r, ~\mathbb P(A)>0 \ \Longrightarrow \  \mathbb P(\mathcal{C}_{\alpha}^{n, =r}|A)\rightarrow \alpha . 
\end{equation}
Furthermore, the test is consistent, i.e.
\begin{equation}\label{test 1b}
\mathbb P(\mathcal{C}_{\alpha}^{n, =r} \cap (\Omega_T^r)^c) \rightarrow \mathbb P((\Omega_T^r)^c).
\end{equation}
(ii) The test defined through \eqref{test2} has asymptotic level at most $\alpha$ in the sense that
\begin{equation}\label{test 2a}
A\subset \Omega_T^{\leq r}, ~\mathbb P(A)>0 \ \Longrightarrow \ \limsup_{n\to\infty} \mathbb P(\mathcal{C}_{\alpha}^{n, \leq r}|A)
\le \alpha . 
\end{equation}
Furthermore, the test is consistent, i.e.
\begin{equation}\label{test 2b}
\mathbb P(\mathcal{C}_{\alpha}^{n, \leq r} \cap (\Omega_T^{\leq r})^c) \rightarrow \mathbb P((\Omega_T^{\leq r})^c).
\end{equation}
\end{cor}

\section{Simulations}
\label{sec6}
\setcounter{equation}{0}
\renewcommand{\theequation}{\thesection.\arabic{equation}}

In this section, we want to examine how well the testing procedure for the maximal rank performs in finite samples. The main focus lies on considering the convergence results in \eqref{finalclt}, \eqref{test 1a} and \eqref{test 1b}. Complementing these results, we examine how well the estimator $\widehat{R}(g,h)_T^n$ works to estimate the maximal rank $R_T$ (using the law of large numbers which is implicitly given by \eqref{finalclt}). To this end, we consider the integer-valued modification of $\widehat{R}(g,h)_T^n$ defined as 
\begin{equation}\label{R^int}
\widehat{R}^{\text{int}}(g,h)_T^n:= \max(0,\min(d,\widehat{R}(g,h)_T^n)).
\end{equation} 
We emphasize that due to the rate of convergence of $\Delta_n^{-1/6}$ we expect a worse performance in finite samples in comparison to the simulation study in \cite{JP13} (there, the rate of convergence is $\Delta_n^{-1/2}$).

\subsection{Results}
All processes are simulated on the interval $[0,1]$ and we use four different frequencies $\Delta_n = 10^{-4}, 10^{-5}$, $10^{-6}$, and $10^{-7}$. 
We remark that even the highest frequency $\Delta_n = 10^{-7}$ is nowadays available for liquid assets.
Following the simulation study in \cite{JP13}, we set $\widetilde \sigma = 2 I_d$ and due to \cite{JLMPV09} we use $\theta = 1/3$ for the pre-averaging procedure. This results in window sizes of $k_n = 157, 718$, $3\,333$, and $15\,471$, respectively. We use the weight functions $g,h$ explicitly constructed in Example \ref{ex:weight functions}. We perform 500 repetitions to uncover the finite sample properties. The following quantities are reported
\begin{itemize}
\item
$\Delta_n \colon$ the sampling frequency;
\item
$[1/3du_n] \colon$ the number of big blocks;
\item
the first four moments of the test statistic $\frac{\widehat{R}(g,h)_1^n - R_1}{\sqrt{u_nV(n,1,g,h)}} $ defined at \eqref{finalclt} to check for 
the normal approximation;
\item
$\Omega_1^r\colon$the proportion of rejection for the possible null hypotheses $\Omega_1^r$ with $r \in \{0, \ldots, d\}$ defined at \eqref{omegar} at level $\alpha = 0.05$;
\item
$\widehat{R}^{\text{int}}(g,h)_1^n\colon$the proportion of the event that the estimator $\widehat{R}^{\text{int}}(g,h)_1^n$ defined at \eqref{R^int} coincides with $R_1$.
\end{itemize}

We conduct the simulation study for the cases $d=1,2,3$. For each of them, we examine different models for the semimartingale $X$. We are interested in how robust our testing procedure is with respect to a violation of assumption (A). This can be seen in case 3, respectively, where the volatility $\sigma_t$ is not continuous. For the noise part, we always assume the covariance structure of $\Sigma= 0.0005 I_d$.

\subsubsection{$d=1$}
We consider the following four models:
\begin{itemize}
\item[(i)]
Model 1: We have vanishing drift $b_t=0$ and constant volatility $\sigma_t=1$, implying $R_1=1$.
\item[(ii)]
Model 2: We observe pure noise, so $b_t=0$ and $\sigma_t=0$, implying $R_1=0$.
\item[(iii)]
Model 3: We have a constant drift of $b_t = 1$ and a volatility of $\sigma_t = 1_{\{t\le0.5\}}$, implying $R_1=1$.
\item[(iv)] 
Model 4: We have a drift of $b_t = 1 + \sin(2t\pi)$ and a volatility of $\sigma_t = \cos(2t\pi)$, implying $R_1=1$.
\end{itemize}
The results for the four models are summarized in the following table according to their order:
\begin{center}
\begin{tabular}{c|c||c|c|c|c||c|c||c }
$\Delta_n$& $[1/3du_n]$	& 1st mt	& 2nd mt 	& 3rd mt 	& 4th mt 	& $\Omega_1^0$ 	& $\Omega_1^1$ 	& $\widehat{R}^{\text{int}}(g,h)_1^n$ \\[1mm]
\hline\hline
$10^{-4}$	& \phantom{1}21	& -0.134 	& 1.460 	& -0.363 	& \phantom{1}6.567	& 0.378	& 0.106 	& 0.734\\
$10^{-5}$	& \phantom{1}46	&  -0.075 	& 1.268	& -0.579	& \phantom{1}5.343	& 0.652	& 0.084 	& 0.862\\
$10^{-6}$	& 100			&  -0.036	& 1.089	& -0.153	& \phantom{1}3.665	& 0.934	& 0.068 	& 0.950\\
$10^{-7}$	& 215			&  -0.068	& 1.097	& -0.079	& \phantom{1}3.932	& 0.998	& 0.060 	& 0.990 \\[1mm]
\hline
$10^{-4}$	& \phantom{1}21	&  -0.056			& 1.403	& -0.403			& \phantom{1}6.206	& 0.096	& 0.454	& 0.790\\
$10^{-5}$	& \phantom{1}46	&  \phantom{-}0.020	& 1.323	& \phantom{-}0.171	& \phantom{1}5.998	& 0.084	& 0.672	& 0.860\\
$10^{-6}$	& 100			&  \phantom{-}0.006	& 1.129	& \phantom{-}0.168	& \phantom{1}3.913	& 0.062	& 0.926	& 0.952\\
$10^{-7}$	& 215			&  -0.016			& 1.024	& -0.039			& \phantom{1}3.536	& 0.048	& 1.000	& 0.992 \\[1mm]
\hline
$10^{-4}$	& \phantom{1}21	& -0.252	& 2.018	& -2.222	& 17.611			& 0.298	& 0.132	& 0.664 \\
$10^{-5}$	& \phantom{1}46	& -0.050	& 1.364	& -0.477	& \phantom{1}6.220	& 0.484	& 0.092	& 0.802 \\
$10^{-6}$	& 100			& -0.140	& 1.105	& -0.427	& \phantom{1}3.596	& 0.708	& 0.076	& 0.878 \\
$10^{-7}$	& 215			& -0.026	& 1.013	& -0.159	& \phantom{1}3.100	& 0.940	& 0.054	& 0.966 \\[1mm]
\hline
$10^{-4}$	& \phantom{1}21	& -0.373	& 1.959	& -3.042	& 15.488			& 0.310	& 0.154	& 0.684 \\
$10^{-5}$	& \phantom{1}46	& -0.231	& 1.389	& -1.271	& \phantom{1}6.902	& 0.484	& 0.076	& 0.788 \\
$10^{-6}$	& 100			& -0.115	& 1.080	& -0.489	& \phantom{1}3.475	& 0.808	& 0.058	& 0.912 \\
$10^{-7}$	& 215			& -0.048	& 0.985	& -0.257	& \phantom{1}3.129	& 0.986	& 0.052	& 0.982
\end{tabular}
\end{center}

\subsubsection{$d=2$}
We consider the following four models:
\begin{itemize}
\item[(i)]
Model 1: We have vanishing drift $b_t = 0$ and constant volatility $\sigma_t = I_2$, implying $R_1=2$.
\item[(ii)]
Model 2: We have pure noise, so $R_1=0$.
\item[(iii)]
Model 3: We have a drift of $b_t = \begin{pmatrix} \phantom{-}1 \\ -1\end{pmatrix}$, and a volatility of $\sigma_t = \begin{pmatrix}1_{\{t\le 0.5\}} & 0 \\ 0 & 1_{\{t> 0.5\}}\end{pmatrix}$, implying $R_1=1$.
\item[(iv)]
Model 4: We have a drift of $b_t = 
\begin{pmatrix}
1+ \sin(2t\pi) \\
1+ \cos(2t\pi) 
\end{pmatrix}$,
and a volatility of\\
$\sigma_t = 
\begin{pmatrix}
\cos(2t\pi) & \cos(2t\pi)  \\
\sin(2t\pi)  & \sin(2t\pi)  
\end{pmatrix}$,
implying $R_1=1$.
\end{itemize}
The results for the four models are summarized in the following table according to their order:

\resizebox{\textwidth}{!}{
\begin{tabular}{c|c||c|c|c|c||c|c|c||c }
$\Delta_n$& $[1/3du_n]$	& 1st mt	& 2nd mt 	& 3rd mt 	& 4th mt 	& $\Omega_1^0$ 	& $\Omega_1^1$ & $\Omega_1^2$ 	& $\widehat{R}^{\text{int}}(g,h)_1^n$ \\[1mm]
\hline\hline
$10^{-4}$	& \phantom{1}10	& -0.146	& 2.836	& \phantom{1}-1.952	& \phantom{1}33.933& 0.524	& 0.264	& 0.226 	& 0.616\\
$10^{-5}$	& \phantom{1}23	& -0.034	& 1.892	& \phantom{1}-0.119	& \phantom{1}13.661& 0.690	& 0.336	& 0.140	& 0.678\\
$10^{-6}$	& \phantom{1}50	& -0.033	& 1.379	& \phantom{1}-0.081	& \phantom{11}5.987	& 0.890	& 0.448	& 0.088	& 0.774\\
$10^{-7}$	& 107			& -0.059	& 1.238	& \phantom{1}-0.263	& \phantom{11}4.451	& 0.992	& 0.626	& 0.078 & 0.842 \\[1mm]
\hline
$10^{-4}$	& \phantom{1}10	& -0.097			& 3.839	& -15.410			& 297.391				& 0.212	& 0.316	& 0.558	& 0.646 \\
$10^{-5}$	& \phantom{1}23	& -0.074			& 2.054	& \phantom{1}-0.181	& \phantom{2}24.391	& 0.130	& 0.328	& 0.720	& 0.716\\
$10^{-6}$	& \phantom{1}50	& \phantom{-}0.113 	& 1.357	& \phantom{-1}0.775	& \phantom{29}6.365	& 0.104	& 0.390	& 0.926	& 0.780\\
$10^{-7}$	& 107			& -0.066			& 1.290	& \phantom{1}-0.274	& \phantom{29}4.658	& 0.084	& 0.672	& 0.992	& 0.882\\[1mm]
\hline
$10^{-4}$	& \phantom{1}10	& -0.293	& 3.747	& \phantom{1}-0.119	& 101.560			& 0.288	& 0.246	& 0.388	& 0.286 \\
$10^{-5}$	& \phantom{1}23	& -0.388	& 1.916	& \phantom{1}-1.635	& \phantom{1}11.157	& 0.236	& 0.156	& 0.428	& 0.346\\
$10^{-6}$	& \phantom{1}50	& -0.090	& 1.251	& \phantom{1}-0.396	& \phantom{10}4.483& 0.424	& 0.090	& 0.482	& 0.590 \\
$10^{-7}$	& 107			& -0.152	& 1.214	& \phantom{1}-0.501	& \phantom{10}4.777& 0.602	& 0.082	& 0.704	& 0.740\\[1mm]
\hline
$10^{-4}$	& \phantom{1}10	& -0.062	& 3.308	& \phantom{1}-6.196	& \phantom{1}88.100& 0.294	& 0.206	& 0.304	& 0.260 \\
$10^{-5}$	& \phantom{1}23	& -0.279	& 2.020	& \phantom{1}-2.432	& \phantom{1}17.677& 0.246	& 0.162	& 0.386 	& 0.418\\
$10^{-6}$	& \phantom{1}50	& -0.061	& 1.438	& \phantom{1}-0.481	& \phantom{11}6.609 & 0.434	& 0.098	& 0.450	& 0.522\\
$10^{-7}$	& 107			& \phantom{-}0.008	& 1.148	& \phantom{1}-0.123	& \phantom{11}4.544	& 0.668	& 0.058	& 0.678		& 0.756
\end{tabular}
}

\subsubsection{$d=3$}
We consider the following four models:
\begin{itemize}
\item[(i)]
Model 1: We have vanishing drift $b_t = 0$ and constant volatility $\sigma_t = I_3$. Hence, the maximal rank is $R_1=3$.
\item[(ii)]
Model 2: We have pure noise, so $R_1=0$.
\item[(iii)]
Model 3: We have a drift of $b_t = \begin{pmatrix} \phantom{-}1 \\ -1 \\ \phantom{-}5\end{pmatrix}$, and a volatility of $\sigma_t = \begin{pmatrix}1_{\{t\le 0.5\}} & 0 & 0 \\ 0 & 1_{\{t> 0.5\}} & 0 \\ 0 & 0 & 0 \end{pmatrix}$, implying $R_1=1$.
\item[(iv)]
Model 4: We have a drift of $b_t = 
\begin{pmatrix}
1+ \sin(2t\pi) \\
1+ \cos(2t\pi) \\
0
\end{pmatrix}$,
and a volatility of\\
$\sigma_t = 
\begin{pmatrix}
\cos(2t\pi) & \cos(2t\pi)	& 0 \\
\sin(2t\pi)  & \sin(2t\pi)	& 0 \\
0		& 0			& 1
\end{pmatrix}$,
implying $R_1 = 2$.
\end{itemize}
The results for the four models are summarized in the following table according to their order:

\resizebox{\textwidth}{!}{
\begin{tabular}{c|c||c|c|c|c||c|c|c|c||c }
$\Delta_n$& $[1/3du_n]$	& 1st mt	& 2nd mt 	& 3rd mt 	& 4th mt 	& $\Omega_1^0$ 	& $\Omega_1^1$ & $\Omega_1^2$ 	&  $\Omega_1^3$ 	& $\widehat{R}^{\text{int}}(g,h)_1^n$  \\[1mm]
\hline\hline
$10^{-4}$	& 	\phantom{1}7	& -0.291	& 4.854	& -11.705			& \phantom{2}195.934				& 0.622	& 0.460	& 0.344	& 0.326	& 0.556\\
$10^{-5}$	& 	15			& -0.041	& 2.631	& \phantom{1}-1.521	& \phantom{21}37.908	& 0.764	& 0.538	& 0.288	& 0.192	& 0.636\\
$10^{-6}$	& 	33			& \phantom{-}0.057	& 1.844	& \phantom{1}-0.258	& \phantom{21}10.721	& 0.930	& 0.692	& 0.368	& 0.146	& 0.702 \\
$10^{-7}$	& 	71			& \phantom{-}0.014	& 1.459	&  \phantom{-1}0.116& \phantom{291}6.382	& 0.988	& 0.870	& 0.414  	& 0.098	& 0.776\\[1mm]
\hline
$10^{-4}$	& 	\phantom{1}7	& -0.177			& 5.507	& \phantom{1}-9.964			& \phantom{2}221.460				& 0.318	& 0.370	& 0.500	& 0.674	& 0.594 \\
$10^{-5}$	& 	15			& \phantom{-}0.017	& 2.779	& \phantom{-1}2.579	& \phantom{22}59.508	& 0.174	& 0.314	& 0.570	& 0.792	& 0.652 \\
$10^{-6}$	& 	33			& \phantom{-}0.023	& 1.565	& \phantom{1}-0.003			& \phantom{225}7.529	& 0.122	& 0.308	& 0.708	& 0.940	& 0.706\\
$10^{-7}$	& 	71			& -0.035		& 1.507		& \phantom{1}-0.340			& \phantom{225}6.874	& 0.100	& 0.452	& 0.860  	& 0.990	& 0.754 \\[1mm]
\hline
$10^{-4}$	& 	\phantom{1}7	& -0.365	& 5.509	& -11.808	& \phantom{2}413.241	& 0.294	& 0.278	& 0.378	& 0.562	& 0.202\\
$10^{-5}$	& 	15			& -0.282	& 2.392	& \phantom{1}-1.297	& \phantom{24}21.282& 0.264	& 0.184	& 0.344	& 0.660	& 0.288 \\
$10^{-6}$	& 	33			& -0.333	& 1.701	& \phantom{1}-1.506	& \phantom{242}8.125& 0.254	& 0.130	& 0.418	& 0.818	& 0.392 \\
$10^{-7}$	& 	71			& -0.074	& 1.279	& \phantom{1-}0.026	& \phantom{242}5.215& 0.446	& 0.074	& 0.502  	& 0.948	& 0.574\\[1mm]
\hline
$10^{-4}$	& 	\phantom{1}7	& -0.238		& 9.259	& 70.679				& 2873.157	 		& 0.470	& 0.368	& 0.348	& 0.440 	& 0.190 \\
$10^{-5}$	& 	15			& -0.082		& 3.063	& -3.513		& \phantom{28}56.660	& 0.560	& 0.320	& 0.212	& 0.340	& 0.246 \\
$10^{-6}$	& 	33			& -0.137		& 1.762	& -0.736		& \phantom{285}9.407	& 0.644	& 0.292	& 0.140  	& 0.362	& 0.380 \\
$10^{-7}$	& 	71			& -0.054		& 1.683	& -0.691		& \phantom{285}8.858	& 0.848	& 0.412	& 0.128  	& 0.430	& 0.494
\end{tabular}
}

\subsection{Summary}

According to our theoretical results, the empirical counterpart of the first four moments, level and power seem to converge to their theoretical
analogues as $\Delta_n \to 0$. However, the speed of convergence depends on the particular model and the dimension $d$.

First, we observe that higher moments seem to converge much slower than lower moments. An extreme example is $d=3$, Model 4, where the simulated fourth
moment equals $2873.157$ at frequency $\Delta_n=10^{-4}$. This effect appears to be stronger for higher dimensions. It can be explained by the fact that the true rate of convergence is $[T/3du_n]^{1/2}$ rather than $\Delta_n^{-1/6}$, which decreases when $d$ is growing. Furthermore, there are several 
small order terms in the expansion of the main statistic, which seem to influence the finite sample performance at relatively low frequencies. This is confirmed
by the observation that we get the best simulation results for constant volatility and vanishing drift, where these lower order terms do not appear.

The approximation of power again depends on the complexity of time-varying coefficients of the model and the dimension $d$. Quite intuitively, we observe
a better power performance for alternative hypotheses, which are more distant to the true one. For instance, for $d=3$ and Model 1, where the maximal
rank is 3, the simulated powers for $\Omega_1^2$, $\Omega_1^1$ and  $\Omega_1^0$ at frequency $\Delta_n=10^{-7}$ are
$0.414$, $0.87$ and $0.988$, respectively. Finally, we remark that, although Model 3 ($d=1,2,3$) does not satisfy our assumptions since the volatility
process is not continuous, the power and level performance is well comparable with other simulated models.

\section{Proofs}
\label{sec5}
\setcounter{equation}{0}
\renewcommand{\theequation}{\thesection.\arabic{equation}}

Before presenting the proofs in detail, let us briefly outline the roadmap of this section. In Subsection \ref{sec5.1} we introduce some technical results about expansions of determinants. We justify the asymptotic expansion at \eqref{detexp} and also show some more involved results.

In Subsection \ref{sec5.2}, we show that -- using a standard localization procedure -- we obtain the stochastic decomposition explained in Remark \ref{rem2}. Moreover, we show how the law of the dominating term in the expansion can be expressed in terms of the notation introduced in Subsection \ref{sec: Notation}. 

Subsection \ref{sec5.3} is especially concerned with the proof of Proposition \ref{prop1}. To this end, we perform a very detailed analysis of the terms $\Gamma_r$ and $\Gamma'_r$ introduced at \eqref{Gammar} and their dependency on the weight function $g$. This mainly relies on an application of the Leibniz rule for the calculation of the determinants and a repeated use of the It\^o isometry to calculate the expectations.

Subsection \ref{sec5.4} deals with the proof of the main Theorem \ref{th2} and of Proposition \ref{prop2}. First, we show that -- thanks to the stochastic expansion established in Subsection \ref{sec5.2} -- the main approximation idea motivated in Subsection \ref{sec3.1} works in the stochastic setting. The second main step in the proof of Theorem \ref{th2} is the application of a stable central limit theorem for semimartingales (see e.g. \cite[Theorem IX.7.28 ]{JS02}). Proposition \ref{prop2}(i) follows along the lines of parts of the proof of Theorem \ref{th2}. Proposition \ref{prop2}(ii) follows by Theorem \ref{th2}, Theorem \ref{th1} -- which in tern is a direct consequence of Theorem \ref{th2} -- and Proposition \ref{prop2}(ii) by applying the delta method for stable convergence. Note that this procedure does only work under a proper choice of the pair of weight functions which fulfills Proposition \ref{prop1}(ii).

The proof of Corollary \ref{corollary: testing} is essentially a consequence of the stable convergence at \eqref{finalclt} and is referred to Subsection \ref{sec5.5}.

\subsection{Expansion of determinants}\label{sec5.1}
Due to Subsection \ref{sec3.1}, the key to identifying the unknown rank of a matrix $A \in \mathcal M$ is the matrix perturbation method which results in the expansion at \eqref{detexp}. While we could show the law of large numbers at \eqref{lln} with an expansion like the one at \eqref{detexp}, we need a higher order expansion of the determinant to derive the central limit theorem at \eqref{clt}. Therefore, we shall introduce some additional notation to the one in Subsection \ref{sec3.1} which is similar to the one introduced in \cite{JP13}.

In the sequel, $\|A\|$ denotes the Euclidean norm of a matrix $A\in \mathcal M$. For any positive integer $m\ge1$ we denote with $\mathcal P_m$ the set of all multi-integers $\mathbf p =(p_1,\ldots, p_m)$ with $p_j\ge0$ and $p_1+\cdots+p_m=d$, and $\mathcal I_\mathbf{p}$ the set of all partitions $\I=(I_1, \ldots, I_m)$ of $\{0,\ldots, d\}$ such that $I_j$ contains exactly $p_j$ points.
For $\mathbf p \in \mathcal P_m$, $\I\in\mathcal I_\mathbf{p}$ and $A_1,\ldots, A_m\in \M$, we call $G^\I_{A_1, \ldots, A_m}$ the matrix in $\mathcal M$ whose $i$th column is the $i$th column of $A_j$ when $j\in I_j$.

Due to the multi-linearity property of the determinant we have the following identity for all $A_1, \ldots, A_m \in \mathcal M$
\begin{equation}
\label{eqn: det sum}
\det(A_1 + \cdots+ A_m)= \sum_{\mathbf p \in \mathcal P_m} \sum_{\I\in\mathcal I_p} \det(G^\I_{A_1, \ldots, A_m}).
\end{equation}
For $A,B,C\in\M$ and $r\in \{0,\ldots,d\}$, recalling \eqref{gamma}, we recover the identity
\begin{equation}\label{gamma2}
\gamma_r(A,B)= \sum_{\I\in \mathcal I_{(r,d-r)}}\det(G^\I_{A, B})
\end{equation}
and set
\begin{equation}\label{gamma'}
\gamma_r'(A,B,C):= \sum_{\I\in \mathcal I_{(r,d-r-1,1)}}\det(G^\I_{A, B, C}),
\end{equation}
with the convention that $\gamma_{-1}(A,B)=0$ and $\gamma'_{d}(A,B,C)=0$. Let $A,B,C,D \in \M$ and $\rank(A)\le r$. Using \eqref{eqn: det sum} we obtain the asymptotic expansion
\begin{align}
\label{asyexp}
&\det(A+\lambda B + \lambda^2 C + \lambda^2 D) 
= \lambda^{d-r} \ga_r(A,B) \\
&\qquad+ \lambda^{d-r+1} \big(\ga_{r-1} (A,B) + \ga'_r(A,B,C) + \ga'_r(A,B,D)  \big) + O(\lambda^{d-r+2}) \qquad \text{as} \quad \lambda\downarrow 0. \nonumber
\end{align}
This observation gives rise to the following lemma (see also \cite[Lemma 6.2]{JP13}).

\begin{lem}
There is a constant $K>0$ such that for all $r\in\{0,\ldots, d\}$, all $\lambda \in(0,1]$ and all $A,B,C,D\in\M$ with $\rank(A)\le r$ we have with $\Lambda = \|A\| + \|B\| + \|C\| + \|D\|$
\begin{multline}\label{first inequality}
\left| \det(A+\lambda B + \lambda^2 C + \lambda^2 D) - \lambda^{d-r} \ga_r(A,B) \right.\\
\left. - \lambda^{d-r+1} 
(\ga_{r-1}(A,B) + \ga_r'(A,B,C) )\right| 
\le K \lambda^{r-d+1} \Lambda^{d-1} (\lambda\Lambda + \|D\|),
\end{multline}
\begin{multline}
\label{eqn: inequality}
\Big | \frac{1}{\lambda^{2d-2r}} \det(A+\lambda B + \lambda^2 C + \lambda^2 D)^2 - \ga_r(A,B)^2 \\
-2\lambda \ga_r(A,B)(\ga_{r-1}(A,B) + \ga_r'(A,B,C))\Big | 
\le K\lambda \Lambda^{2d-1} (\lambda\Lambda + \|D\|).
\end{multline}
If further $\lambda'\in(0,1]$, $A',B',C',D'\in\M$ with $\rank(A')\le r$ and $\Lambda' = \|A'\| + \|B'\| + \|C'\| + \|D'\|$, then
\begin{multline}
\label{eqn: (6.6)}
\Big| \frac{1}{(\lambda\lambda')^{2d-2r}} \det(A+\lambda B + \lambda^2 C + \lambda^2 D)^2 \det(A'+\lambda' B' + \lambda'^2 C' + \lambda'^2 D')^2 \\
 -\ga_r(A,B)^2 \ga_r(A',B')^2 \Big| 
 \le  K(\lambda + \lambda')\Lambda\Lambda'. 
\end{multline}
\end{lem}

\begin{proof}
The inequalities at \eqref{first inequality} and \eqref{eqn: (6.6)} essentially follow from the asymptotic expansion at \eqref{asyexp} and the fact that there is a $K>0$ such that for any $\mathbf p \in \mathcal P_4$, $\I \in \mathcal I_{\mathbf p}$ and $\lambda \in (0,1]$ we have $|\det (G^\I_{A,\lambda B, \lambda^2 C, \lambda^2 D})| = \lambda^{p_2 + 2p_3 + 2p_4} |\det (G^\I_{A,B,C,D})|\le K \lambda^{p_2 + 2p_3 + 2p_4} \Lambda^{d-p_4} \|D\|^{p_4}$. \eqref{eqn: inequality} follows from \eqref{first inequality} by taking squares.
\end{proof}

\subsection{The stochastic decomposition}\label{sec5.2}
Under assumption (A) and by a standard localization procedure (see e.g. \cite[Section 3]{BGJPS06}), it is no restriction to make the following technical assumption.

\nib Assumption (A1): \rm Assumption (A) holds and the processes $X_t$, $b_t$, $\sigma_t$, $a_t$, $v_t$, $a'_t$, $v'_t$, $a''_t$, $v''_t$ defined at \eqref{X} and \eqref{bsigma} are uniformly bounded in $(\omega, t)$.
\qed
\vsq

We make the convention that all constants are denoted by $K$, or $K_p$ if they depend on an additional parameter $p$. The constants never depend on $T,t,n, i,j$. To ease notation, we use generic constants that may change from line to line. 
We introduce the filtration $\mathcal (\mathcal H_t)_{t\in[0,T]}$ defined as
\[
\mathcal H_t := \mathcal F_t \vee \sigma(\varepsilon),
\]
where $\sigma(\varepsilon)$ is the $\sigma$-field generated by the whole process $(\ep)_{t\in[0,T]}$.
For any process $V$ and for the filtrations $(\mathcal F_t)_{t\in[0,T]}$, $(\mathcal H_t)_{t\in[0,T]}$ and $\kappa =1,2$, we will use the simplifying notation
\begin{equation} \label{eqn: filtrations}
V_i^{n,\kappa} = V_{(3i +\kappa-1)du_n}, \qquad \mathcal F_i^{n,\kappa} = \mathcal F_{(3i +\kappa-1)du_n}, \qquad \mathcal H_i^{n,\kappa} = \mathcal H_{(3i +\kappa-1)du_n}.
\end{equation}
Note that we have the `nesting property' $\mathcal F_i^{n,1}\subset \mathcal F_i^{n,2}$ and $\mathcal H_i^{n,1} \subset \mathcal H_i^{n,2}$, respectively.
Now, we show that under (A) we can obtain the stochastic decomposition at \eqref{stoch decomp} explained in Remark \ref{rem2}. To do so, we notice (see \cite[Section 6]{JP13}) that under (A), and for any $z\le t\le s$, we have the following expansion for the increment 
$X_s-X_t = \int_t^s b_u\,ds + \int_t^s \sigma_s\,dW_s$
(using vector notation):
\begin{align*}
&\int_t^s b_u\,du = \alpha_1 + \alpha_2 + \alpha_3 + \alpha_4, \\
&\int_t^s \sigma_u\,dW_u= \alpha_5 +\alpha_6 + \alpha_7 + \alpha_8 + \alpha_9 + \alpha_{10} +\alpha_{11},
\end{align*}
where
\begin{align}
\label{eqn: expansion 4}
& \alpha_1 = b_z(s-t),& 
&\alpha_2= \int_t^s \left(\int_z^u a'_w\,dw\right)\,du, \\
&\alpha_3= v'_z\int_t^s (W_u-W_z)\, du, & \nonumber
& \alpha_4 = \int_t^s \left(\int_z^u (v'_w-v'_z)\,dW_w\right)\,du, \\
& \alpha_5= \sigma_z(W_s-W_t),&\nonumber
&\alpha_6= a_z\int_t^s (u-z)\,dW_u,\\
&\alpha_7= \int_t^s\left(\int_z^u(a_w-a_z)\,dw\right)\,dW_u, &\nonumber
& \alpha_8= v_z\int_t^s(W_u-W_z)\,dW_u,\\
&\alpha_9= \int_t^s\left(\int_z^u\left(\int_z^w a''_r\,dr\right)\,dW_w \right)\, dW_u, &\nonumber
& \alpha_{10}= v''_z \int_t^s\left(\int_z^u(W_w-W_z)\,dW_w\right)\,dW_u, \\
& \alpha_{11}= \int_t^s\left(\int_z^u \left(\int_z^w(v''_r - v''_z)\,dW_r\right)\,dW_w\right)\,dW_u.\nonumber
\end{align}
By the Burkholder-Gundy inequality (see e.g. \cite{Yor}) we have under (A1) for all $p,t,s>0$ and for $V= X,\sigma, b,v$ that 
\begin{equation}\label{Burkholder}
\E[\sup_{u\in[0,s]} \|V_{t+u} - V_t \|^p\,|\,\mathcal F_t]\le K_p \, s^{p/2}.
\end{equation}
We set
\begin{equation*}
\eta_{t,s} = \sup_{u\in[0,s], \ V=a,v', v''} \|V_{t+u} - V_t\|^2,
\quad
n_i^{n,\kappa} = \sqrt{\E[\eta_{(3i +\kappa-1)du_n, \kappa du_n}\,|\,\mathcal F_i^{n,\kappa}]}.
\end{equation*}
Using the Burkholder-Gundy inequality and H\"older inequality leads to (recall that $z\le s$)
\begin{align}
\label{eqn: estimates}
	\E[\|\alpha_j\|^p\,|\,\mathcal F_z]\le
	\begin{cases}
		K_p(s-z)^{p/2} & \text{if } j=5 \\
		K_p(s-z)^{p} & \text{if } j=1,8 \\
		K_p(s-z)^{3p/2} & \text{if } j=3,6,10 \\
		K_p(s-z)^{2p} & \text{if } j=2,9 \\
		K_p(s-z)^{3p/2} \,\E[\eta_{z,s-z}^p \,|\, \mathcal F_z]  & \text{if } j=4,7,11.
	\end{cases}
\end{align}

Let $g$ be a weight function (see Subsection \ref{sec3.2}) and $g^n$ its discretization introduced at \eqref{g^n}. For $\kappa = 1,2$ we define the function
\[
g_{n,\kappa}(x):= g^n(\kappa u_n x) = \sum_{j=1}^{k_n-1} g\left(\frac{j}{k_n}\right) \mathbf{1}_{(\kappa(j-1)\Delta_n, \kappa j \Delta_n]}(x).
\]
Using \eqref{eqn: expansion 4} with $z=(3i + \kappa - 1)du_n, t= ((3i + \kappa - 1)d+\kappa(j-1))u_n, s = ((3i + \kappa - 1)d+\kappa j)u_n$  with $i \in\{ 0,\ldots, [T/3d u_n]-1\}, j \in\{1,\ldots, d\}$ and $\kappa=1,2$ we then obtain the stochastic decomposition at \eqref{stoch decomp}, namely
\begin{align*}
&\frac{1}{\sqrt{\kappa u_n}} \text{mat} \left(\overline Z(g)^{n,\kappa}_{(3i + \kappa - 1)dk_n}, \cdots, \overline Z(g)^{n,\kappa}_{((3i + \kappa - 1)d+\kappa(d-1))k_n}\right)\\[1.5 ex]
& = A(g)_{i}^{n,\kappa}
+ \sqrt{\kappa u_n} \left(B(1,g)_{i}^{n,\kappa} + B(2,g)_{i}^{n,\kappa} + B(3,g)_{i}^{n,\kappa} \right) + \kappa u_n C(g)_{i}^{n,\kappa} + \kappa u_n D(g)_{i}^{n,\kappa},  
\end{align*}
where (using vector notation)
\begin{align*}
	A(g)_{i,j}^{n,\kappa} 
	&= \frac{\sigma_i^{n,\kappa}}{\sqrt{\kappa u_n}}
		\int_{((3i + \kappa - 1)d+\kappa(j-1))u_n}^{((3i + \kappa - 1)d+\kappa j)u_n}
		g_{n,\kappa}(s - ((3i + \kappa - 1)d+\kappa(j-1))u_n)
		\,d W_s, \\
	B(1,g)_{i,j}^{n,\kappa}
	&= 	\frac{b_i^{n,\kappa}}{\kappa u_n}
		\int_{((3i + \kappa - 1)d+\kappa(j-1))u_n}^{((3i + \kappa - 1)d+\kappa j)u_n}
		g_{n,\kappa}(s - ((3i + \kappa - 1)d+\kappa(j-1))u_n)
		\,ds \\
	&\hspace{-4em}
		+ \frac{v_i^{n,\kappa}}{\kappa u_n}
		\int_{((3i + \kappa - 1)d+\kappa(j-1))u_n}^{((3i + \kappa - 1)d+\kappa j)u_n}
		g_{n,\kappa}(s - ((3i + \kappa - 1)d+\kappa(j-1))u_n)\\
	&	\quad \times	\left(W_s - W_{(3i + \kappa - 1)du_n}\right) dW_s,\\
	B(2,g)_{i,j}^{n,\kappa}
	&= \frac{\widetilde\sigma}{\sqrt{\kappa u_n}}
		\int_{((3i + \kappa - 1)d+\kappa(j-1))u_n}^{((3i + \kappa - 1)d+\kappa j)u_n}
		g_{n,\kappa}(s - ((3i + \kappa - 1)d+\kappa(j-1))u_n)
		\,d W'_s, \\
	B(3,g)_{i,j}^{n,\kappa}
	&=	 \frac{1}{\kappa u_n} \overline \ep(g)^{n,\kappa}_{((3i + \kappa - 1)d+\kappa(j-1))k_n} , \\
	C(g)_{i,j}^{n,\kappa}
	&= 	\frac{a_i^{n,\kappa}}{(\kappa u_n)^{3/2}}
		\int_{((3i + \kappa - 1)d+\kappa(j-1))u_n}^{((3i + \kappa - 1)d+\kappa j)u_n}
		g_{n,\kappa}(s - ((3i + \kappa - 1)d+\kappa(j-1))u_n)\\
	&	\quad \times	
		\left(s -(3i + \kappa - 1)du_n\right) dW_s  \\
	&\hspace{-4em}	
		+\frac{{v'}_i^{n,\kappa}}{(\kappa u_n)^{3/2}}
		\int_{((3i + \kappa - 1)d+\kappa(j-1))u_n}^{((3i + \kappa - 1)d+\kappa j)u_n}
		g_{n,\kappa}(s - ((3i + \kappa - 1)d+\kappa(j-1))u_n)\\
	&	\quad \times 	\left(W_s - W_{(3i + \kappa - 1)du_n}\right) ds \\
	&\hspace{-4em}	
		+\frac{{v''}_i^{n,\kappa}}{(\kappa u_n)^{3/2}}
		\int_{((3i + \kappa - 1)d+\kappa(j-1))u_n}^{((3i + \kappa - 1)d+\kappa j)u_n}
		g_{n,\kappa}(s - ((3i + \kappa - 1)d+\kappa(j-1))u_n) \\
	&	\quad \times \left(\int_{(3i + \kappa - 1)du_n}^s \left(W_u-W_{(3i + \kappa - 1)du_n}\right)dW_u\right) dW_s,
\end{align*}
and $D(g)_{i,j}^{n,\kappa}$ is the remainder term. In the sequel, we will make the convention that $B(g)_{i}^{n,\kappa} := B(1,g)_{i}^{n,\kappa} + B(2,g)_{i}^{n,\kappa} + B(3,g)_{i}^{n,\kappa}$. With the following lemma, we can deduce that under assumption (A1) the $\R^{d\times d}$-valued sequences $ A(g)_{i}^{n,\kappa}, B(g)_{i}^{n,\kappa}, C(g)_{i}^{n,\kappa}, D(g)_{i}^{n,\kappa}$ are tight (see also equation (6.15) in \cite{JP13}).

\begin{lem}
\label{lemma: boundedness}
Let the assumptions (A1) and (E) be statisfied.
For $p\ge 1$ there is a $K_p>0$ such that we have the following estimates
\begin{align*}
	&\E\left[\|A(g)_{i,j}^{n,\kappa} \|^p
	+\|B(g)_{i,j}^{n,\kappa} \|^p\,
	+\|C(g)_{i,j}^{n,\kappa} \|^p \,
	\big|\,\mathcal F_i^{n,\kappa}\right]\le K_p, \\
	&\E\left[\|D(g)_{i,j}^{n,\kappa} \|^p\,
	\big|\,\mathcal F_i^{n,\kappa}\right]
	\le K_p\left(u_n^{p/2} + (\eta_i^{n,\kappa})^{p\wedge 2} \right)
	\le K_p.
\end{align*}
\end{lem}

\begin{proof}
To show the estimate for the term $B(3,g)_{i,j}^{n,\kappa}$, we refer to \cite[Equation (16.2.3)]{JP12} which implies that $\E[\|\overline \ep(g)^{n,\kappa}_{i}\|^p]\le K_p k_n^{-p/2}$, such that the claim follows by recalling \eqref{kn}.
For the remaining terms we use \eqref{eqn: estimates} with $z=(3i + \kappa - 1)du_n, t= ((3i + \kappa - 1)d+\kappa(j-1))u_n, s = ((3i + \kappa - 1)d+\kappa j)u_n$ plus the fact that $g_{n,\kappa}$ is uniformly bounded in $n$.
\end{proof}

\begin{lem}\label{lemma: convergence}
Assume (A1) and (E). Then $u_n \E\left[\sum_{i=0}^{[T/3du_n]-1} \eta_i^{n,\kappa}\right]\to0$.
\end{lem}

\begin{proof}
The proof follows along the lines of the proof of Lemma 6.3 in \cite{JP13}.
\end{proof}

\begin{lem}\label{lemma: law}
Let the assumptions (A1) and (E) be satisfied.
Fix a weight function $g$. Then, for any $r\in\{0,\ldots, d\}$, $i=0,\ldots, [T/3du_n]-1$ and $\kappa=1,2$ the $\mathcal F_i^{n,\kappa}$-conditional law of 
\[
\gamma_r(A(g)_i^{n,\kappa},B(g)_i^{n,\kappa})^2
\] 
coincides with the $\mathcal F_i^{n,\kappa}$-conditional law of 
\[
\overline F_r(\sigma_i^{n,\kappa}, \widetilde \sigma, v_i^{n,\kappa}, b_i^{n,\kappa}, \Sigma^n, g^n,\kappa),
\]
where 
\begin{equation}\label{Sigma^n}
\Sigma^n := \frac{\theta^3}{k_n^3 \Delta_n^2}\Sigma.
\end{equation}
\end{lem}

\begin{proof}
The quantity  $ \overline F_r(\sigma_i^{n,\kappa}, \widetilde \sigma, v_i^{n,\kappa}, b_i^{n,\kappa}, \Sigma^n, g^n,\kappa)$ can be realized on $(\Omega, \mathcal F, (\mathcal F_t)_{t\in [0,T]}, \mathbb P)$ by taking 
\[
\overline W_t = \frac{W_{((3id +\kappa-1)d+ t)u_n} - W_{(3id +\kappa-1)du_n}}{\sqrt{u_n}}, 
\quad \overline W'_t = \frac{W'_{((3id +\kappa-1)d+ t)u_n} - W'_{(3id +\kappa-1)du_n}}{\sqrt{u_n}}.
\]
Then define $\overline\Theta$ implicitly, such that 
\[
\frac{1}{ \kappa} \left( \frac{\psi_1(g)}{\theta^3}\right)^{1/2} \sum_{m=1}^d \left((\Sigma^n)^{1/2}\right)^{lm}\,\overline\Theta_{\kappa j}^m = B(3,g)_{i,j}^{n,\kappa}.
\]
Indeed, we know that $B(3,g)_{i,j}^{n,\kappa}$ is a centered Gaussian random variable with covariance matrix
\[
\frac{1}{\kappa^2 u_n^2} \sum_{\mu=0}^{k_n-1} \left(g\left(\frac{\mu+1}{k_n}\right) - g\left(\frac{\mu}{k_n}\right)\right)^2 \Sigma
= \frac{\psi_1(g^n)}{\kappa^2 k_n^3 \Delta_n^2} \Sigma
= \frac{1}{\kappa^2} \frac{\psi_1(g^n)}{\theta^3}\Sigma^n,
\]
where $\psi_1(g^n)$ is defined at \eqref{psi_1(gn)}. We also remark that due to \eqref{kn} we obtain that $\Sigma^n= (1 + o(\Delta_n^{1/6}))\Sigma$.
\end{proof}
As a direct consequence of Lemma \ref{lemma: law} we can deduce that
\begin{align}\label{conditional exp}
\E[\gamma_r(A(g)_i^{n,\kappa},B(g)_i^{n,\kappa})^2 \,|\, \mathcal F_i^{n,\kappa} ] &= \Gamma_r(\sigma_i^{n,\kappa}, \widetilde \sigma, v_i^{n,\kappa}, b_i^{n,\kappa}, \Sigma^n, g^n,\kappa), \\
\Var[\gamma_r(A(g)_i^{n,\kappa},B(g)_i^{n,\kappa})^2 \,|\, \mathcal F_i^{n,\kappa} ] &= \Gamma'_r(\sigma_i^{n,\kappa}, \widetilde \sigma, v_i^{n,\kappa}, b_i^{n,\kappa}, \Sigma^n, g^n,\kappa).\label{conditional var}
\end{align}

\subsection{Proof of Lemma \ref{lemma:Gamma} and Proposition \ref{prop1}} \label{sec5.3}

\begin{proof}[Proof of Lemma \ref{lemma:Gamma}]
The proof of Lemma \ref{lemma:Gamma} follows along the lines of the proof of \cite[Lemma 3.1]{JP13}. (Notice that we can incorporate the additional terms with $\overline \Theta_j$ appearing in \eqref{Psi2} in the terms $\Theta_j$ at \cite[Equation (6.8)]{JP13}.) 
\end{proof}

\begin{proof}[Proof of Proposition \ref{prop1}]
We start with the proof of part (i). Let $r\in\{0,\ldots, d\}$, $\underline u = (\alpha, \beta, \gamma, a, \varphi) \in \mathcal U$, $\kappa=1,2$ and $g$ be any weight function. Using the notation at \eqref{Psi1} and \eqref{Psi2}, we define the matrices
\[
A(\underline u, g, \kappa):= \left(\Psi(\underline u, g, \kappa)_i^j\right)_{i,j=1,\ldots, d}, \qquad
B(\underline u, g, \kappa):= \left(\Psi(\underline u, g, \kappa)_i^{d+j}\right)_{i,j=1,\ldots, d},
\]
being elements of $\mathcal M$.
Furthermore, for $\I\in\mathcal I_{(r,d-r)}$ we will use the notation 
\[
G^\I_{A(\underline u, g, \kappa), B(\underline u, g, \kappa)} = \left(\left(G^\I_{A(\underline u, g, \kappa), B(\underline u, g, \kappa)}\right)_i^j\right)_{i,j=1,\ldots, d}.
\]
Then, developing the determinant with the Leibniz rule, we obtain the identity
\begin{align}\label{determinant}
&\Gamma_r (\underline u, g,\kappa)
=\overline \E\left[ \gamma_r(A(\underline u, g, \kappa), B(\underline u, g, \kappa))^2\right] \\
&=\overline \E\left[ \sum_{\I, \I' \in \mathcal I_{(r,d-r)}} \det\left(G^\I_{A(\underline u, g, \kappa), B(\underline u, g, \kappa)}\right) \det\left(G^{\I'}_{A(\underline u, g, \kappa), B(\underline u, g, \kappa)}\right) \right] \nonumber\\
&=\overline \E\left[ \sum_{\I, \I' \in \mathcal I_{(r,d-r)}} 
\sum_{\pi, \pi'\in \mathfrak S_d} \sgn(\pi)\sgn(\pi')
\prod_{i=1}^d \left(G^{\I}_{A(\underline u, g, \kappa), B(\underline u, g, \kappa)}\right)_i^{\pi(i)}\prod_{j=1}^d \left(G^{\I'}_{A(\underline u, g, \kappa), B(\underline u, g, \kappa)}\right)_j^{\pi'(j)}  \right]\nonumber \\
&= \sum_{\I, \I' \in \mathcal I_{(r,d-r)}} 
\sum_{\pi, \pi'\in \mathfrak S_d}\sgn(\pi)\sgn(\pi')
\prod_{i=1}^d 
\overline \E \left[\left(G^{\I}_{A(\underline u, g, \kappa), B(\underline u, g, \kappa)}\right)_i^{\pi(i)} \left(G^{\I'}_{A(\underline u, g, \kappa), B(\underline u, g, \kappa)}\right)_i^{\pi'(i)}\right],\nonumber
\end{align}
where $\mathfrak S_d$ denotes the group of all permutations of the set $\{1,\ldots,d\}$ and $\sgn(\pi)\in\{-1,1\}$ is the sign of the permutation $\pi\in \mathfrak S_d$. The last step in the computation is due to the fact that the vectors $\left(G^{\I}_{A(\underline u, g, \kappa), B(\underline u, g, \kappa)}\right)_i$ and $\left(G^{\I'}_{A(\underline u, g, \kappa), B(\underline u, g, \kappa)}\right)_j$ are uncorrelated if $i\neq j$. Thus, for fixed $r\in\{0,\ldots, d\}$ and $\kappa=1,2$, the mapping $(\underline u,g)\mapsto\Gamma_r (\underline u, g,\kappa)$ can be considered as a polynomial in $\binom{d}{r}^2\times(d\,!)^2 \times d$
variables of the form
\begin{equation}
\label{eqn: variables}
\overline \E \left[\left(G^{\I}_{A(\underline u, g, \kappa), B(\underline u, g, \kappa)}\right)_i^{\pi(i)} \left(G^{\I'}_{A(\underline u, g, \kappa), B(\underline u, g, \kappa)}\right)_i^{\pi'(i)}\right],
\end{equation}
where $\I, \I'\in\mathcal I_{(r,d-r)}$, $\pi, \pi'\in \mathfrak S_d$ and $i=1,\ldots, d$.
Using It\^o's isometry, \eqref{eqn: variables} takes one of the following three forms with $l,l'\in\{1,\ldots,d\}$:
\begin{align*}
&\overline \E \left[  \Psi(\underline u, g, \kappa)_i^{l} \, \Psi(\underline u, g, \kappa)_i^{l'}\right]
= \psi_2(g) \sum_{m=1}^q \al^{lm}\al^{l'm},\\
&\overline \E \left[  \Psi(\underline u, g, \kappa)_i^{l} \, \Psi(\underline u, g, \kappa)_i^{d+l'}\right]
=0, \\
&\overline \E \left[  \Psi(\underline u, g, \kappa)_i^{d+l} \, \Psi(\underline u, g, \kappa)_i^{d+l'}\right]
= \psi_3(g)^2 a^l a^{l'} + \big(\psi_4(g) + (i-1)\psi_2(g)\big) \sum_{m,k=1}^q \gamma^{lkm} \gamma^{l'km} \\
&+ \psi_2(g)  \sum_{m=1}^d \beta^{lm}\beta^{l'm} + \frac{\psi_1(g)}{\kappa^2\theta^3}  \sum_{m=1}^d \left(\varphi^{1/2}\right)^{lm}\left(\varphi^{1/2}\right)^{l'm}.
\end{align*}

Hence, if we additionally fix $\underline u\in \mathcal U$, then there is a polynomial $ \tau_{r,\underline u,\kappa}\colon \R^4 \to \R$ such that the mapping $g\mapsto \Gamma_r(\underline u,g,\kappa)$ can be written as 
\[
g\mapsto \tau_{r,\underline u,\kappa}(\psi_1(g),\psi_2(g),\psi_3(g),\psi_4(g)).
\]
This shows the first part of \eqref{psidep}.
To show the second part, we use the relationship
\[
\Gamma'_{r}(\underline u,  g,\kappa)
	=\overline \E\left[ \gamma_r(A(\underline u, g, \kappa), B(\underline u, g, \kappa))^4\right]
	- \Gamma_{r}(\underline u, g,\kappa)^2.
\]
By a similar calculation as in \eqref{determinant} we obtain that 
\begin{multline*}
\overline \E\left[ \gamma_r(A(\underline u, g, \kappa), B(\underline u, g, \kappa))^4\right] \\
= \sum_{\I, \I', \I'', \I''' \in \mathcal I_{(r,d-r)}} 
\sum_{\pi, \pi',\pi'', \pi'''\in \mathfrak S_d}\sgn(\pi)\sgn(\pi')\sgn(\pi'')\sgn(\pi''')  
\prod_{i=1}^d 
\overline \E \left[\left(G^{\I}_{A(\underline u, g, \kappa), B(\underline u, g, \kappa)}\right)_i^{\pi(i)} \right. \\
\left.\times
\left(G^{\I'}_{A(\underline u, g, \kappa), B(\underline u, g, \kappa)}\right)_i^{\pi'(i)}
\left(G^{\I''}_{A(\underline u, g, \kappa), B(\underline u, g, \kappa)}\right)_i^{\pi''(i)}
\left(G^{\I'''}_{A(\underline u, g, \kappa), B(\underline u, g, \kappa)}\right)_i^{\pi'''(i)}
\right].
\end{multline*}
If we fix again $r\in\{0,\ldots, d\}$ and $\kappa=1,2$, the mapping $(\uu, g)\mapsto \overline \E\left[ \gamma_r(A(\underline u, g, \kappa), B(\underline u, g, \kappa))^4\right]$ can be considered as a polynomial in $\binom{d}{r}^4\times(d\,!)^4 \times d$
variables of the form
\begin{equation}
\label{eqn: variables 2}
\overline \E \left[\left(G^{\I}_{A(\underline u, g, \kappa), B(\underline u, g, \kappa)}\right)_i^{\pi(i)} 
\left(G^{\I'}_{A(\underline u, g, \kappa), B(\underline u, g, \kappa)}\right)_i^{\pi'(i)}
\left(G^{\I''}_{A(\underline u, g, \kappa), B(\underline u, g, \kappa)}\right)_i^{\pi''(i)}
\left(G^{\I'''}_{A(\underline u, g, \kappa), B(\underline u, g, \kappa)}\right)_i^{\pi'''(i)}
\right],
\end{equation}
where $\I, \I', \I'', \I'''\in\mathcal I_{(r,d-r)}$, $\pi, \pi', \pi'', \pi'''\in \mathfrak S_d$ and $i=1,\ldots, d$.
By a careful calculation, we can see that \eqref{eqn: variables 2} takes one of the following five forms with $l,l', l'', l'''\in\{1,\ldots,d\}$:
\begin{align*}
&\overline \E \left[  \Psi(\underline u, g, \kappa)_i^{l} \, \Psi(\underline u, g, \kappa)_i^{l'}
\, \Psi(\underline u, g, \kappa)_i^{l''}\, \Psi(\underline u, g, \kappa)_i^{l'''}\right]
= \psi_2(g)^2 K_\al,\\
&\overline \E \left[  \Psi(\underline u, g, \kappa)_i^{l} \, \Psi(\underline u, g, \kappa)_i^{l'}
\, \Psi(\underline u, g, \kappa)_i^{l''}\, \Psi(\underline u, g, \kappa)_i^{d+l'''}\right]
=0, \\
&\overline \E \left[  \Psi(\underline u, g, \kappa)_i^{l} \, \Psi(\underline u, g, \kappa)_i^{l'}
\, \Psi(\underline u, g, \kappa)_i^{d+l''}\, \Psi(\underline u, g, \kappa)_i^{d+l'''}\right]
=\psi_2(g) \sum_{m=1}^q \al^{lm}\al^{l'm} \\
&\quad\times
\left(\psi_3(g)^2 a^{l''} a^{l'''} + \big(\psi_4(g) + (i-1)\psi_2(g)\big) \sum_{m,k=1}^q \gamma^{l''km} \gamma^{l'''km}\right. \\
&\qquad\qquad\left.+ \psi_2(g)  \sum_{m=1}^d \beta^{l''m}\beta^{l'''m} + \frac{\psi_1(g)}{\kappa^2\theta^3}  \sum_{m=1}^d \left(\varphi^{1/2}\right)^{l''m}\left(\varphi^{1/2}\right)^{l'''m}
\right), \\
&\overline \E \left[  \Psi(\underline u, g, \kappa)_i^{l} \, \Psi(\underline u, g, \kappa)_i^{d+l'}
\, \Psi(\underline u, g, \kappa)_i^{d+l''}\, \Psi(\underline u, g, \kappa)_i^{d+l'''}\right]
=0,\\
&\overline \E \left[  \Psi(\underline u, g, \kappa)_i^{l} \, \Psi(\underline u, g, \kappa)_i^{l'}
\, \Psi(\underline u, g, \kappa)_i^{d+l''}\, \Psi(\underline u, g, \kappa)_i^{d+l'''}\right]
=\psi_3(g)^4 a^la^{l'}a^{l''}a^{l'''} \\
&+ \big(\psi_4(g) + (i-1)\psi_2(g)\big)^2 K_\gamma + \psi_2(g)^2 K_\beta +\big( \frac{\psi_1(g)}{\kappa^2\theta^3}\big)^2 K_\varphi
+ \psi_3(g)^2 \big(\psi_4(g) + (i-1)\psi_2(g)\big) K_{a,\gamma}   \\ 
&+\psi_3(g)^2 \psi_2(g) K_{a,\beta} + \psi_3(g)^2\frac{\psi_1(g)}{\kappa^2\theta^3} K_{a,\varphi} +\big(\psi_4(g) + (i-1)\psi_2(g)\big)\psi_2(g) K_{\gamma, \beta} \\
&+ \big(\psi_4(g) + (i-1)\psi_2(g)\big) \frac{\psi_1(g)}{\kappa^2\theta^3} K_{\gamma, \varphi} + \psi_2(g)\frac{\psi_1(g)}{\kappa^2\theta^3} K_{\beta, \varphi}.
\end{align*} 
We remark that the constants do not depend on $\kappa$. 
Consequently, if we additionally fix $\uu\in \mathcal U$, there is a polynomial $\tau'_{r,\uu,\kappa}\colon \R^4\to\R$ such that the mapping $g\mapsto \Gamma'_r(\uu, g, \kappa)$ can be written as 
\[
g\mapsto \tau'_{r,\underline u,\kappa}(\psi_1(g),\psi_2(g),\psi_3(g),\psi_4(g)),
\]
which proves part (i) of Proposition \ref{prop1}. By an inspection of the previous calculations, we see that the only term where $\kappa$ appears is the term $\frac{\psi_1(g)}{\kappa^2\theta^3}$. Hence, for any $r\in\{0,\ldots, d\}$, $\uu\in \mathcal U$, we have 
\begin{align*}
\tau_{r,\uu, 1}(x_1,x_2,x_3,x_4) &= \tau_{r,\uu,2}(4x_1,x_2,x_3,x_4), \\
\tau'_{r,\uu, 1}(x_1,x_2,x_3,x_4) &= \tau'_{r,\uu,2}(4x_1,x_2,x_3,x_4), \qquad (x_1,x_2,x_3,x_4)\in\R^4.
\end{align*}
This shows part (ii) of Proposition \ref{prop1}.
\end{proof}

\subsection{Proof of Theorem \ref{th2} and Proposition \ref{prop2}} \label{sec5.4}
Let $g$ be a weight function. We begin by constructing approximations for the main test statistics $S(g)_T^{n,\kappa}$ defined at \eqref{main-statistic} and $V(g,h)_T^{n,\kappa\kappa'}$ given at \eqref{V(g)}, \eqref{V(h)} and \eqref{V(g,h)}.
The approximations are discretized versions of $S(r,g)_T^\kappa$ (see \eqref{lln}) and the right-hand sides of \eqref{V(g,h)_T^{n,11}} to \eqref{V(g,h)_T^{n,12}}
\begin{align}\nonumber
S(r,g)_T^{n,\kappa} &:= 3du_n \sum_{i=0}^{[T/3du_n]} \gamma_r(A(g)_i^{n,\kappa},B(g)_i^{n,\kappa})^2, \\
 \label{eqn: V}
V(r,g,h)_T^{n,\kappa\kappa'} &:= 
\begin{cases}
9d^2 u_n \sum_{i=0}^{[T/3du_n]-1} \gamma_r(A(g)_i^{n,1},B(g)_i^{n,1})^4, & \text{if } \kappa = \kappa'=1, \\
9d^2 u_n \sum_{i=0}^{[T/3du_n]-1} \gamma_r(A(h)_i^{n,2},B(h)_i^{n,2})^4, & \text{if } \kappa = \kappa'=2, \\
9d^2 u_n \sum_{i=0}^{[T/3du_n]-1} \gamma_r(A(g)_i^{n,1},B(g)_i^{n,1})^2  \\
\quad \times \gamma_r(A(h)_i^{n,2},B(h)_i^{n,2})^2, & \text{if } \kappa =1,  \kappa'=2. 
\end{cases}
\end{align}
The lemma is based on the asymptotic expansion at \eqref{asyexp}.

\begin{lem}\label{lemma: 5.4}
Assume (A1), (E), let $r\in\{0,\ldots, d\}$, $\kappa, \kappa'=1,2$ and $g,h$ be two weight functions (not necessarily satisfying the conditions of Proposition \ref{prop1}(ii)). Then, on $\Omega_T^{\le r}$, we have that 
\begin{gather}\label{(6.18)}
\frac{1}{\sqrt{u_n}}\left( \frac{1}{(\kappa u_n)^{d-r}}S(g)_T^{n,\kappa} - S(r,g)_T^{n,\kappa} \right) \toop 0, \\ \label{(6.19)}
\frac{1}{(\kappa\kappa' u_n^2)^{d-r}}V(g,h)_T^{n,\kappa\kappa'} - V(r,g,h)_T^{n,\kappa\kappa'}  \toop 0.
\end{gather}
\end{lem}

\begin{proof}
The proof is an adaption of the proof of \cite[Lemma 6.4]{JP13}. Let $\xi(g)_i^{n,\kappa}$ denote the $i$th summand on the right-hand side of \eqref{main-statistic}. We start by showing \eqref{(6.18)}. To this end, we use the fact that $\rank(A(g)_i^{n,\kappa})\le r$ for all $i$ to apply the inequality at \eqref{eqn: inequality} with $\lambda = \sqrt{\kappa u_n}$ to obtain
\begin{align*}
\frac{1}{(\kappa u_n)^{d-r}} \xi(g)_i^{n,\kappa} = \gamma_r(A(g)_i^{n,\kappa},B(g)_i^{n,\kappa})^2 + 2\sqrt{\kappa u_n} \zeta(g)_i^{n,\kappa} + \widetilde \zeta(g)_i^{n,\kappa},
\end{align*}
where with the Cauchy-Schwarz inequality (using the conventions after \eqref{gamma'})
\begin{gather*}
\zeta(g)_i^{n,\kappa}:= \gamma_r(A(g)_i^{n,\kappa}, B(g)_i^{n,\kappa})(\gamma_{r-1}(A(g)_i^{n,\kappa}, B(g)_i^{n,\kappa}) + \gamma'_r(A(g)_i^{n,\kappa}, B(g)_i^{n,\kappa}, C(g)_i^{n,\kappa})), \\
\E\left[|\widetilde \zeta(g)_i^{n,\kappa}|\right]\le Ku_n + K\sqrt{u_n}\,\E[\eta_i^{n,\kappa}].
\end{gather*}
Applying Lemma \ref{lemma: convergence}, we deduce that $\sqrt{u_n} \sum_{i=0}^{[T/3du_n]-1} \widetilde \zeta(g)_i^{n,\kappa} \toop 0$. Regarding the structure of $S(r,g)_T^{n,\kappa}$ we need to prove that $u_n \sum_{i=0}^{[T/3du_n]-1} \zeta(g)_i^{n,\kappa} \toop 0$. To this end, we consider the decomposition $\zeta(g)_i^{n,\kappa} = \zeta'(g)_i^{n,\kappa} + \zeta''(g)_i^{n,\kappa}$, where $\zeta''(g)_i^{n,\kappa} = \E[\zeta(g)_i^{n,\kappa}\,|\,\mathcal F_i^{n,\kappa}]$. We obtain
\begin{align}\label{Doob}
\begin{split}
\E\left[\left(u_n \sum_{i=0}^{[T/3du_n]-1} \zeta''(g)_i^{n,\kappa} \right)^2\right]
&= u_n^2 \sum_{i,j=0}^{[T/3du_n]-1} \E[ \zeta''(g)_i^{n,\kappa}\zeta''(g)_j^{n,\kappa}] \\
&= u_n^2 \sum_{i=0}^{[T/3du_n]-1} \E[ |\zeta''(g)_i^{n,\kappa}|^2] \\ 
&\le u_n KT\to0,
\end{split}
\end{align}
where the second identity follows from the fact that $\zeta''(g)_i^{n,\kappa}$ is $\mathcal F_{i+1}^{n,\kappa}$-measurable and the last estimate is a consequence of Lemma \ref{lemma: boundedness} and the fact that $\gamma_r$ and $\gamma'_r$ are continuous functions. Hence, we know that $u_n \sum_{i=0}^{[T/3du_n]-1} \zeta''(g)_i^{n,\kappa} \toop 0$. So it is sufficient to show that $\zeta'(g)_i^{n,\kappa}=0$, or the even stronger result that 
\begin{equation}\label{=0}
\E[\zeta(g)_i^{n,\kappa}\,|\,\mathcal H_i^{n,\kappa}\vee \sigma(W')]=0,
\end{equation} 
where $\sigma(W')$ is the $\sigma$-field generated by the whole process $W'$ and $\mathcal H_i^{n,\kappa}$ was introduced before and in \eqref{eqn: filtrations}. Recalling the definitions at \eqref{gamma2} and \eqref{gamma'}, equation \eqref{=0} follows by the implication
\begin{gather}
\I\in \mathcal I_{(r,d-r)}, \ \I'\in \mathcal I_{(r-1,d-r+1)}, \ \I'' \in \mathcal I_{(r,d-r-1,1)}
\ \Longrightarrow \notag\\
\label{eqn: 5.20 b}
\E\left[ \det \left( G^{\I}_{A(g)_i^{n,\kappa}, B(g)_i^{n,\kappa}} \right) 
\det \left( G^{\I'}_{A(g)_i^{n,\kappa}, B(g)_i^{n,\kappa}} \right)
\,|\,\mathcal H_i^{n,\kappa} \vee \sigma(W') \right] = 0, \\
\E\left[ \det \left( G^{\I}_{A(g)_i^{n,\kappa}, B(g)_i^{n,\kappa}} \right) 
\det \left( G^{\I''}_{A(g)_i^{n,\kappa}, B(g)_i^{n,\kappa}, C(g)_i^{n,\kappa}} \right)
\,|\,\mathcal H_i^{n,\kappa} \vee \sigma(W') \right] = 0.\label{eqn: 5.20 c}
\end{gather}
Note that due to the conventions after \eqref{gamma'} the left-hand side of \eqref{eqn: 5.20 b} is 0 if $r=0$, and the left-hand side of \eqref{eqn: 5.20 c} is 0 if $r=d$. The $d$-dimensional variables $A(g)_{i,j}^{n,\kappa}$, $B(g)_{i,j}^{n,\kappa}$ and $C(g)_{i,j}^{n,\kappa}$ can be written in the form 
\begin{equation*}
\Phi \left(\omega, (W(\omega)_{(3i + \kappa-1)du_n + t} - W(\omega)_{(3i + \kappa-1)du_n})_{t\ge 0} \right),
\end{equation*}
where $\Phi$ is a $(\mathcal H_i^{n,\kappa} \vee  \sigma(W')) \otimes \mathcal C^q$-measurable function on $\Omega \times C(\R_+, \R^q)$. Here $C(\R_+, \R^q)$ is the set of all continuous functions on $\R_+$ with values in $\R^q$ and $\mathcal C^q$ is its Borel $\sigma$-field for the local uniform topology. Notice that for $ \Phi = A(g)_{i,j}^{n,\kappa}$ or $\Phi = C(g)_{i,j}^{n,\kappa}$, the mapping $x \mapsto \Phi(\omega, x)$ is odd, meaning that $\Phi(\omega, -x) = -\Phi(\omega, x)$, and for $\Phi = B(g)_{i,j}^{n,\kappa}$, it is even, meaning that  $\Phi(\omega, -x) = \Phi(\omega, x)$.
We set
\begin{align*}
\Psi &= \det \left( G^{\I}_{A(g)_i^{n,\kappa}, B(g)_i^{n,\kappa}} \right), \\
\Psi' &= \det \left( G^{\I'}_{A(g)_i^{n,\kappa}, B(g)_i^{n,\kappa}} \right), \\
\Psi'' &= \det \left( G^{\I''}_{A(g)_i^{n,\kappa}, B(g)_i^{n,\kappa}, C(g)_i^{n,\kappa}} \right),
\end{align*} 
where $\Psi, \Psi', \Psi''$ are functions similar to $\Phi$. Due to the multilinearity of the determinant
we can deduce that if $r$ is even, then $\Psi$ is even and $\Psi'$, $\Psi''$ are odd. If $r$ is odd, $\Psi$ is odd and $\Psi'$,$\Psi''$ are even. Thus, in all cases, the products $\Psi\Psi'$ and $\Psi\Psi''$ are odd. Now, the $(\mathcal H_i^{n,\kappa} \vee \sigma(W'))$-conditional law of 
$\left( W_{(3i + \kappa-1)du_n + t} - W_{(3i + \kappa-1)du_n}\right)_{t\ge 0}$ is invariant under the map $x \mapsto -x$ on $C(\R_+, \R^q)$, which implies \eqref{eqn: 5.20 b}, and hence \eqref{=0}.

The proof of \eqref{(6.19)} is more direct. We apply the estimate at \eqref{eqn: (6.6)} with $\lambda = \sqrt{\kappa u_n}, \lambda' = \sqrt{\kappa' u_n}$. With the previous notation and Lemma \ref{lemma: boundedness} we obtain
\begin{align*}
\begin{split}
&\E\left[\left|\frac{1}{(\kappa \kappa'u_n^2)^{d-r}}V(g,h)_T^{n,\kappa\kappa'} - V(r,g,h)_T^{n,\kappa\kappa'} \right|\right] \\
&\le 9d^2 u_n \sum_{i=0}^{[T/3du_n]-1} \E\left[\left| \frac{1}{(\kappa \kappa'u_n^2)^{d-r}} \xi(g)_i^{n,\kappa}\xi(h)_i^{n,\kappa'}  - \gamma'_r(A(g)_i^{n,\kappa}, B(g)_i^{n,\kappa})^2 \gamma'_r(A(h)_i^{n,\kappa'}, B(h)_i^{n,\kappa'})^2 \right| \right] \\
&\le 9d^2 KT \sqrt{u_n} \to0,
\end{split}
\end{align*}
which implies \eqref{(6.19)}.
\end{proof}

With respect to Lemma \ref{lemma: 5.4}, Theorem \ref{th2} follows by showing the following lemma.

\begin{lem}\label{lemma: Jacob b}
Assume (A1), (E). Let $r\in\{0,\ldots, d\}$ and $g,h$ be two weight function satisfying the conditions of Proposition \ref{prop1}(ii). Then, on $\Omega_T^{\le r}$, we have the stable convergence
\begin{equation*}
U'(r,g,h)_T^{n} \stab \mathcal{MN}(0,V(r,g,h)_T),
\end{equation*}
where $V(r,g,h)_T$ is defined at \eqref{V} and the two-dimensional statistic $U'(r,g,h)_T^{n} = (U'(r,g,h)_T^{n,1}, U'(r,g,h)_T^{n,2})$ is given via
\begin{equation*}
U'(r,g,h)_T^{n} := \frac{1}{\sqrt{u_n}} \left(  S(r,g)_T^{n,1} - S(r,g)_T^{1}, 
S(r,h)_T^{n,2} - S(r,h)_T^{2}  \right). 
\end{equation*}
\end{lem}
We will do the proof of Lemma \ref{lemma: Jacob b} in three steps: 

(i) Recall that due to Proposition \ref{prop1}(ii) we have that $S(r,g)_T^1= S(r,h)_T^2$. By a Riemann approximation argument, one can show that 
\begin{gather}\label{eqn: riemann}
\frac{1}{\sqrt{u_n}} \left( 3du_n \sum_{i=0}^{[T/3du_n]-1} \Gamma_r(\sigma_i^{n,1}, \widetilde \sigma, v_i^{n,1}, b_i^{n,1}, \Sigma, g,1) - \int_0^T \Gamma_r(\sigma_s, \widetilde \sigma, v_s, b_s, \Sigma, g,1)\,ds\right) \toop0, \\
\frac{1}{\sqrt{u_n}} \left( 3du_n \sum_{i=0}^{[T/3du_n]-1} \Gamma_r(\sigma_i^{n,2}, \widetilde \sigma, v_i^{n,2}, b_i^{n,2}, \Sigma, h,2) - \int_0^T \Gamma_r(\sigma_s, \widetilde \sigma, v_s, b_s, \Sigma, h,2)\,ds\right) \toop0. \nonumber
\end{gather}
More precisely, we use the fact that for a fixed weight function $g$ and $\kappa=1,2$, the map $\mathcal{U} \ni \underline u \mapsto \Gamma_r(\uu,g,\kappa)$ is a polynomial (and hence $C^\infty$) as well as the fact that thanks to assumption (A) the processes $\sigma$, $v$ and $b$ are It\^o semimartingales and hence c\`adl\`ag (see section 8 in \cite{BGJPS06} for more details).

(ii) We identify the limit by proving that
\begin{gather} \label{eqn: identification1}
3d \sqrt{u_n}  \sum_{i=0}^{[T/3du_n]-1} \left(\Gamma_r(\sigma_i^{n,1}, \widetilde \sigma, v_i^{n,1}, b_i^{n,1}, \Sigma^n, g^n,1) - \Gamma_r(\sigma_i^{n,1}, \widetilde \sigma, v_i^{n,1}, b_i^{n,1}, \Sigma, g,1) \right) \toop 0, \\
3d \sqrt{u_n}  \sum_{i=0}^{[T/3du_n]-1} \left(\Gamma_r(\sigma_i^{n,2}, \widetilde \sigma, v_i^{n,2}, b_i^{n,2}, \Sigma^n, h^n,2) - \Gamma_r(\sigma_i^{n,2}, \widetilde \sigma, v_i^{n,2}, b_i^{n,2}, \Sigma, h,2) \right) \toop 0.\label{eqn: identification2}
\end{gather}

(iii) We prove the stable convergence 
\begin{equation}\label{eqn: stable}
U''(r,g,h)_T^{n} \stab \mathcal{MN}(0,V(r,g,h)_T),
\end{equation}
for the two-dimensional statistic $U''(r,g,h)_T^{n} = (U''(r,g,h)_T^{n,1}, U''(r,g,h)_T^{n,2})$ with components 
\begin{align*}
U''(r,g,h)_T^{n,1}&=3d\sqrt{u_n}  \sum_{i=0}^{[T/3du_n]-1} \left( \gamma_r(A(g)_i^{n,1},B(g)_i^{n,1})^2 - \Gamma_r(\sigma_i^{n,1}, \widetilde \sigma, v_i^{n,1}, b_i^{n,1}, \Sigma^n, g^n,1)\right), \\
 U''(r,g,h)_T^{n,2}&=3d\sqrt{u_n}  \sum_{i=0}^{[T/3du_n]-1} \left( \gamma_r(A(h)_i^{n,2},B(h)_i^{n,2})^2 - \Gamma_r(\sigma_i^{n,2}, \widetilde \sigma, v_i^{n,2}, b_i^{n,2}, \Sigma^n, h^n,2)\right).
\end{align*}
The following lemma is concerned with the convergence at \eqref{eqn: identification1} and \eqref{eqn: identification2}, respectively. 

\begin{lem}\label{lemma: identification}
Assume (A1), (E). Let $r\in\{0,\ldots, d\}$, $\kappa=1,2$ and $g$ be a weight function. Then, on $\Omega_T^{\le r}$, it holds that
\begin{equation}\label{eqn: identification2b}
 \sqrt{u_n}  \sum_{i=0}^{[T/3du_n]-1} \left(\Gamma_r(\sigma_i^{n,\kappa}, \widetilde \sigma, v_i^{n,\kappa}, b_i^{n,\kappa}, \Sigma^n, g^n,\kappa) - \Gamma_r(\sigma_i^{n,\kappa}, \widetilde \sigma, v_i^{n,\kappa}, b_i^{n,\kappa}, \Sigma, g,\kappa) \right) \toop 0.
\end{equation}
\end{lem}

\begin{proof}
Fix $r\in\{0,\ldots, d\}$, $\kappa=1,2$ and a weight function $g$. Recall that by Proposition \ref{prop1}(i) for any $\uu \in \mathcal U$ there is a polynomial $\tau_{r,\uu,\kappa}$ such that $\Gamma_r(\uu,g,\kappa) = \tau_{r,\uu,\kappa}(\psi_1(g), \ldots, \psi_4(g))$. An inspection of the proof of Proposition \ref{prop1}(i) yields that the map
\[
\mathcal U \times \R^4 \to \R, \quad
(\al,\beta, \gamma, a, \Sigma, \psi_1(g), \ldots, \psi_4(g)) \mapsto \tau_{r,(\alpha,\beta, \gamma,a, \Sigma),\kappa}(\psi_1(g), \ldots, \psi_4(g))
\]
is a $C^\infty$-function. Consider the first order partial derivatives in $(\Sigma, \psi_1(g), \ldots, \psi(g))$. For fixed $\Sigma, g$, they are continuous in $(\alpha, \beta, \gamma,a)$. Therefore, by a first order Taylor expansion, we obtain that for any compact set $A\subset \mathcal M' \times \mathcal M \times \R^{dq^2} \times \R^d$
\begin{multline*}
\sup_{(\alpha, \beta, \gamma,a)\in A} | \Gamma_r(\alpha, \beta, \gamma, a, \Sigma^n, g^n, \kappa) - \Gamma_r(\alpha, \beta, \gamma, a, \Sigma, g, \kappa)| \\
 \le K_A \big\|(\Sigma^n, \psi_1(g^n), \ldots, \psi_4(g^n)) - (\Sigma, \psi_1(g), \ldots, \psi_4(g)) \big\|_{\R^{d^2}\times \R^4},
\end{multline*}
where $\| \cdot \|_{\R^{d^2}\times \R^4}$ is the Euclidean norm in $\R^{d^2}\times \R^4$. Combining \eqref{Sigma^n} and \eqref{kn} we get that $(\Sigma^n)_{ij} - \Sigma_{ij} = o(\Delta_n^{1/6})$, $i,j=1,\ldots, d$, and with \eqref{g^n}, \eqref{psi_1(gn)}, we have that $\psi_l(g^n) - \psi_l(g) = O(k_n^{-1})$, $l=1,\ldots, 4$. Again using \eqref{kn} this implies that 
\[
\big\|(\Sigma^n, \psi_1(g^n), \ldots, \psi_4(g^n)) - (\Sigma, \psi_1(g), \ldots, \psi_4(g)) \big\|_{\R^{d^2}\times \R^4} = o(\Delta_n^{1/6}).
\]
Now, we apply assumption (A1) to deduce that 
\[
\sup_{s\in[0,T]} \E \big[|\Gamma_r(\sigma_s, \widetilde \sigma, v_s, b_s, \Sigma^n, g^n, \kappa) - \Gamma_r(\sigma_s, \widetilde \sigma, v_s, b_s, \Sigma, g, \kappa) |\,\big | \,\mathcal F_s\big] = o(\Delta_n^{1/6}),
\]
and hence
\begin{multline*}
 \sqrt{u_n}  \sum_{i=0}^{[T/3du_n]-1} \E \big[|\Gamma_r(\sigma_i^{n,\kappa}, \widetilde \sigma, v_i^{n,\kappa}, b_i^{n,\kappa}, \Sigma^n, g^n,\kappa)\\ - \Gamma_r(\sigma_i^{n,\kappa}, \widetilde \sigma, v_i^{n,\kappa}, b_i^{n,\kappa}, \Sigma, g,\kappa) |\,\big | \,\mathcal F_i^{n,\kappa} \big] 
  = o\left(\frac{\Delta_n^{1/6}}{\sqrt{u_n}}\right) = o(1),
\end{multline*}
which implies \eqref{eqn: identification2b}.
\end{proof}

The next lemma deals with the stable convergence at \eqref{eqn: stable}.

\begin{lem}\label{lemma: Jacod}
Assume (A1), (E). Let $r\in\{0,\ldots, d\}$ and $g,h$ be two weight function satisfying the conditions of Proposition \ref{prop1}(ii). Then, on $\Omega_T^{\le r}$, we have the stable convergence
\begin{equation*}
U''(r,g,h)_T^{n} \stab \mathcal{MN}(0,V(r,g,h)_T),
\end{equation*}
where $V(r,g,h)_T$ is defined at \eqref{V} and the two-dimensional statistic $U''(r,g,h)_T^{n} = (U''(r,g,h)_T^{n,1}, U''(r,g,h)_T^{n,2})$ is given after \eqref{eqn: stable}.
\end{lem}

\begin{proof}
We apply a simplified version of Theorem IX.7.28 in \cite{JS02}. To this end, we introduce the two-dimensional variables $\xi_i^n$ with components
\begin{align*}
\xi_i^{n,1} &= 3d\sqrt{u_n}  \left( \gamma_r(A(g)_i^{n,1},B(g)_i^{n,1})^2 - \Gamma_r(\sigma_i^{n,1}, \widetilde \sigma, v_i^{n,1}, b_i^{n,1}, \Sigma^n, g^n,1)\right), \\
\xi_i^{n,2}&=3d\sqrt{u_n} \left( \gamma_r(A(h)_i^{n,2},B(h)_i^{n,2})^2 - \Gamma_r(\sigma_i^{n,2}, \widetilde \sigma, v_i^{n,2}, b_i^{n,2}, \Sigma^n, h^n,2)\right).\nonumber 
\end{align*}
We must prove the following five statements where $\kappa, \kappa'=1,2$:
\begin{align}
\label{eqn: Jacod 1 b}
& \sum_{i=0}^{[T/3du_n]-1} \E[\xi_i^{n,\kappa}\,|\,\mathcal F_i^{n,1}] \toop 0, \\
\label{eqn: Jacod 2 b}
& \sum_{i=0}^{[T/3du_n]-1} \E[\xi_i^{n,\kappa} \xi_i^{n,\kappa'} \,|\,\mathcal F_i^{n,1}] 
	\toop V(r,g,h)_T^{\kappa,\kappa'}, \\
\label{eqn: Jacod 3 b}
& \sum_{i=0}^{[T/3du_n]-1} \E[\xi_i^{n,\kappa} ( W^m_{3(i+1)du_n} - W^m_{3idu_n})\,|\,\mathcal F_i^{n,1}] 
	\toop 0, \\
\label{eqn: Jacod 4 b}
& \sum_{i=0}^{[T/3du_n]-1} \E[\|\xi_i^n\|^2 \mathbf{1}_{\{\|\xi_i^n\|> \epsilon\}}\,|\,\mathcal F_i^{n,1}] 
	\toop 0 \qquad \forall \epsilon >0, \\
\label{eqn: Jacod 5 b}
& \sum_{i=0}^{[T/3du_n]-1} \E[\xi_i^{n,\kappa} (N_{3(i+1)du_n} - N_{3idu_n})\,|\,\mathcal F_i^{n,1}] 
	\toop 0,
\end{align}
where $W^m$ is any of the components of $W$ and $N$ is a one-dimensional bounded martingale, orthogonal to $(W,W')$ in the sense that the covariation between $N$ and $W^m$, as well as the covariation between $N$ and $W'^m$ vanishes. We will later specify the conditions on $N$. If \eqref{eqn: Jacod 1 b} to \eqref{eqn: Jacod 5 b} hold, then Theorem IX.7.28 in \cite{JS02} yields that 
\begin{equation*}
U''(r,g,h)_T^{n} \stab \mathcal U''(r,g,h)_T,
\end{equation*}
where the random random variable $\mathcal U''(r,g,h)_T$ is defined on an extension $(\widetilde{\Omega}, \widetilde{\mathcal{F}}, \widetilde{\mathbb P})$ 
of the original probability space $(\Omega, \mathcal{F}, \mathbb P)$. It can be realized as
\begin{equation}\label{limit}
\mathcal U''(r,g,h)_T = \int_0^T y_s d W'_s + \int_0^T z_s d \widetilde W_s,
\end{equation}
where $\widetilde W$ is a $d$-dimensional Brownian motion independent of $\mathcal F$, and -- for fixed $\widetilde \sigma, \Sigma, g,h$ -- $y$ and $z$ are c\`adl\`ag processes with values in $\R^{d\times d}$ which are adapted to the filtration generated by $\sigma, b,v$. Moreover, $y$ and $z$ can be characterized by 
\begin{equation*}
\sum_{i=0}^{[T/3du_n]-1} \E[\xi_i^{n,\kappa} ( W'^m_{3(i+1)du_n} - W'^m_{3idu_n})\,|\,\mathcal F_i^{n,1}] 
	\toop \int_0^T y^m_s d s,
\end{equation*}
and
\begin{equation*}
V(r,g,h)_T = \int_0^T \big(y_sy_s^\star + z_s z_s^\star \big)ds.
\end{equation*}
Since $\widetilde W$ and $W'$ are independent of $\mathcal G$ and $y,z$ are $\mathcal G$-measurable, \eqref{limit} yields that $\mathcal U''(r,g,h)_T$ is mixed normal with $\mathcal G$-conditional mean 0 and $\mathcal G$-conditional covariance $V(r,g,h)_T$. Now, we turn to the proof of \eqref{eqn: Jacod 1 b} to \eqref{eqn: Jacod 5 b}.

(i) We use equation \eqref{conditional exp} to derive that $\E[\xi_i^{n,\kappa}\,|\,\mathcal F_i^{n,\kappa}] =0$ for $\kappa=1,2$. Using the nesting property $\mathcal F_i^{n,1} \subseteq \mathcal F_i^{n,2}$ and the tower property, we immediately obtain \eqref{eqn: Jacod 1 b}.

(ii) With equation \eqref{conditional var} one can show that
\begin{align*}\nonumber
\E[\xi_i^{n,1} \xi_i^{n,1} \,|\,\mathcal F_i^{n,1}] 
&= 9d^2 u_n \Gamma'_r(\sigma_i^{n,1}, \widetilde \sigma, v_i^{n,1}, b_i^{n,1}, \Sigma^n, g^n,1), \\
\E[\xi_i^{n,2} \xi_i^{n,2} \,|\,\mathcal F_i^{n,2}] 
&= 9d^2 u_n \Gamma'_r(\sigma_i^{n,2}, \widetilde \sigma, v_i^{n,2}, b_i^{n,2}, \Sigma^n, h^n,2).\label{cond exp}
\end{align*}
Now, we have to carefully evaluate the term $\E[  \Gamma'_r(\sigma_i^{n,2}, \widetilde \sigma, v_i^{n,2}, b_i^{n,2}, \Sigma^n, h^n,2)\,|\, \mathcal F_i^{n,1}]$.
Recall \eqref{Burkholder} which implies that 
\[
\sup_{V=\sigma, v,b} \E[\| V_i^{n,2} - V_i^{n,1} \| \,|\,\mathcal F_i^{n,1} ] \le K\sqrt{u_n}.
\]
Using the multi-linearity property of the determinant and the fact that $\Gamma'_r$ consists of determinants to the power four, we end up with
\begin{align*}
&\E[  \Gamma'_r(\sigma_i^{n,2}, \widetilde \sigma, v_i^{n,2}, b_i^{n,2}, \Sigma^n, h^n,2)\,|\, \mathcal F_i^{n,1}] \\
&= \E[  \Gamma'_r(\sigma_i^{n,1} +(\sigma_i^{n,2} - \sigma_i^{n,1}), \widetilde \sigma, v_i^{n,1} + (v_i^{n,2} - v_i^{n,1}), b_i^{n,1} + (b_i^{n,2} - b_i^{n,1}), \Sigma^n, h^n,2)\,|\, \mathcal F_i^{n,1}] \\
&= \Gamma'_r(\sigma_i^{n,1}, \widetilde \sigma, v_i^{n,1}, b_i^{n,1}, \Sigma^n, h^n,2)+ O(u_n^2).
\end{align*}
Hence, 
\[
\sum_{i=0}^{[T/3du_n]-1} \left(\E[\xi_i^{n,2}\xi_i^{n,2}\,|\,\mathcal F_i^{n,1}] - \Gamma'_r(\sigma_i^{n,1}, \widetilde \sigma, v_i^{n,1}, b_i^{n,1}, \Sigma^n, h^n,2)\right) \toop 0.
\]
Since $\xi_i^{n,1}$ is $\mathcal F_i^{n,2}$-measurable, $\mathcal F_i^{n,1} \subseteq \mathcal F_i^{n,2}$ and $\E[\xi_i^{n,2}\,|\,\mathcal F_i^{n,2}] =0$, we can deduce that $\E[\xi_i^{n,1} \xi_i^{n,2} \,|\,\mathcal F_i^{n,1}] =0$.
It follows along the lines of the proof of Lemma \ref{lemma: identification} that 
\begin{gather*}
9d^2 u_n  \sum_{i=0}^{[T/3du_n]-1} \left( \Gamma'_r(\sigma_i^{n,1}, \widetilde \sigma, v_i^{n,1}, b_i^{n,1}, \Sigma^n, g^n,1) - \Gamma'_r(\sigma_i^{n,1}, \widetilde \sigma, v_i^{n,1}, b_i^{n,1}, \Sigma, g,1)\right) \toop 0, \\
9d^2 u_n  \sum_{i=0}^{[T/3du_n]-1} \left( \Gamma'_r(\sigma_i^{n,1}, \widetilde \sigma, v_i^{n,1}, b_i^{n,1}, \Sigma^n, h^n,2) - \Gamma'_r(\sigma_i^{n,1}, \widetilde \sigma, v_i^{n,1}, b_i^{n,1}, \Sigma, h,2)\right) \toop 0.
\end{gather*}
And by a Riemann approximation argument similar to the one used to show \eqref{eqn: riemann}, one can deduce that
\begin{gather*}
9d^2 u_n  \sum_{i=0}^{[T/3du_n]-1} \Gamma'_r(\sigma_i^{n,1}, \widetilde \sigma, v_i^{n,1}, b_i^{n,1}, \Sigma, g,1) - 3d  \int_0^T \Gamma'_r(\sigma_s, \widetilde{\sigma}, v_s,b_s, \Sigma ,  g,1) ds \toop 0, \\
9d^2 u_n  \sum_{i=0}^{[T/3du_n]-1} \Gamma'_r(\sigma_i^{n,1}, \widetilde \sigma, v_i^{n,1}, b_i^{n,1}, \Sigma, h,2) - 3d  \int_0^T \Gamma'_r(\sigma_s, \widetilde{\sigma}, v_s,b_s, \Sigma ,  h,2) ds \toop 0
\end{gather*}
which gives \eqref{eqn: Jacod 2 b}.

(iii) We will show \eqref{eqn: Jacod 3 b} by proving that 
\begin{gather}\label{eqn: Jacod 3 b proof 1}
\E[\xi_i^{n,1} (W^m_{3(i+1)du_n} - W^m_{3idu_n})\,|\,\mathcal H_i^{n,1}\vee \sigma(W')] =0, \\
\E[\xi_i^{n,2} (W^m_{3(i+1)du_n} - W^m_{(3i+1)du_n})\,|\,\mathcal H_i^{n,2}\vee \sigma(W')] =0.\label{eqn: Jacod 3 b proof 2}
\end{gather}
Indeed, for $\kappa=1$, \eqref{eqn: Jacod 3 b proof 1} directly implies \eqref{eqn: Jacod 3 b}. For $\kappa=2$, we use the relationship
\begin{multline*}
\E[\xi_i^{n,1} (W^m_{3(i+1)du_n} - W^m_{3idu_n})\,|\,\mathcal F_i^{n,1}] 
= \E\big[ (W^m_{(3i+1)du_n} - W^m_{3idu_n} \E[ \xi_i^{n,2} \,|\, \mathcal F_i^{n,2}]
	\,\big|\,\mathcal F_i^{n,1}\big] \\
 + \E\big[ \E[ (W^m_{3(i+1)du_n} - W^m_{(3i+1)du_n} \xi_i^{n,2} \,|\, \mathcal H_i^{n,2}\vee \sigma(W')] \,\big|\,\mathcal F_i^{n,1}\big] .
\end{multline*}
Since $\E[ \xi_i^{n,2} \,|\, \mathcal F_i^{n,2}]=0$, showing \eqref{eqn: Jacod 3 b proof 2} implies \eqref{eqn: Jacod 3 b} in this case.
Similar to the proof of Lemma \ref{lemma: 5.4} one 
can write $\xi_i^{n,\kappa}$ as function of the form
\begin{equation*}
\label{eqn: function Phi}
\Phi \left(\omega, (W(\omega)_{(3i + \kappa-1)du_n + t} - W(\omega)_{(3i + \kappa-1)du_n})_{t\ge 0}\right),
\end{equation*}
where $\Phi$ is a $(\mathcal H_i^{n,\kappa}\vee \sigma(W')) \otimes \mathcal C^q $-measurable function on $\Omega \times C(\R_+, \R^q)$. We have already seen that $A(g)_{i,j}^{n,\kappa}$ and $B(g)_{i,\kappa}^{n,1}$ can also be considered as function of the form \eqref{eqn: function Phi} where $A(g)_{i,j}^{n,\kappa}$ is an odd function and $B(g)_{i,j}^{n,\kappa}$ is an even function. Since  $\xi_i^{n,\kappa}$ consists of squared determinants, the function $\Phi$ in \eqref{eqn: function Phi} is always even in the sense that $\Phi(\omega,(x_1,\ldots, x_q)) = \Phi(\omega,-(x_1,\ldots, x_q))$, no matter if $r$ is even or odd. Consequently the map
\[
(x_1,\ldots, x_q) \mapsto x_m \Phi(\omega,(x_1,\ldots, x_q))
\]
is odd such that \eqref{eqn: Jacod 3 b proof 1}, \eqref{eqn: Jacod 3 b proof 2} follow by a standard argument. 

(iv) Lemma \ref{lemma: boundedness} implies that $\E[\| \xi_i^n\|^4\,|\, \mathcal F_i^{n,1}] \le Ku_n^2$,
such that \eqref{eqn: Jacod 4 b} follows by a standard argument.

(v) The proof of \eqref{eqn: Jacod 5 b} is somewhat more involved than the previous steps. First, we introduce two filtrations: $(\mathcal F^{(0)}_t)_{t\in[0,T]}$ which is generated by all processes appearing in assumption (A) plus the Brownian motion $W'$. In contrast, the filtration $(\mathcal F^{(1)}_t)_{t\in[0,T]}$ is generated by the noise process $\ep$ only. Note that due to assumption (E), $\mathcal F^{(0)}_t$ and $\mathcal F^{(1)}_t$ are independent. Following the proof of \cite[Lemma 5.7]{JLMPV09}, it is sufficient to show \eqref{eqn: Jacod 5 b} for all one-dimensional bounded martingales in a set $\mathcal N = \mathcal N^0 \cup \mathcal N^1$. Here, $\mathcal N^0$ consists of all $(\mathcal F^{(0)}_t)$-martingales which are orthogonal to $(W,W')$. The set $\mathcal N^1$ comprises all $(\mathcal F^{(1)}_t)$-L\'evy-martingales $N$, such that there exists an integer $m\ge1$, time points $0\le t_1 < \cdots <t_m \le T$ and a bounded Borel-function $\tilde f\colon (\R^d)^m \to \R$ with the relation
\begin{equation}\label{eqn: rep 1}
N_t = \E[N_\infty\,|\, \mathcal F^{(1)}_t ] , \qquad N_\infty = \tilde f(\ep_{t_1}, \ldots, \ep_{t_m}).
\end{equation}
Let $N\in\mathcal N^0$. With a similar argumentation like in point (iii), \eqref{eqn: Jacod 5 b} follows by proving that 
\begin{equation}\label{eqn: Jacod 5 b proof}
\E[\xi_i^{n,\kappa} (N_{3(i+1)du_n} - N_{(3i+\kappa-1)du_n})\,|\,\mathcal H_i^{n,\kappa}] =0.
\end{equation}
By assumption, $N$ is independent of $\ep$ so $N$ is also orthogonal to $(W, W')$ conditionally on $\mathcal H_i^{n,\kappa}$. The variable $\xi_i^{n,\kappa}$ can be considered as a $\mathcal H_i^{n,\kappa} \otimes \mathcal C^q\otimes \mathcal C^d$-measurable function on $\Omega\times C(\R_+,\R^q)\times C(\R_+,\R^d)$ of the form
\begin{equation*}
\Phi \left(\omega, (W(\omega)_{(3i+\kappa-1)du_n + t} - W(\omega)_{(3i+\kappa-1)du_n})_{t\ge 0}, 
(W'(\omega)_{(3i+\kappa-1)du_n + t} - W'(\omega)_{(3i+\kappa-1)du_n})_{t\ge 0}\right).
\end{equation*}
By virtue of the representation theorem (see \cite[Proposition V.3.2]{Yor}), we can -- conditionally on $\mathcal H_i^{n,\kappa}$ -- write $\Phi$ as the sum of a constant and a stochastic integral over the interval $((3i+\kappa-1)du_n, 3(i+1)du_n]$ with respect to $(W,W')$ for a suitable $(q+d)$-dimensional predictable integrand. Then, thanks to the It\^o-isometry and the fact that the covariation of $N$ and any component of $(W,W')$ vanishes, one ends up with \eqref{eqn: Jacod 5 b proof}.

\noindent Now, let $N\in \mathcal N^1$ with the representation \eqref{eqn: rep 1}. If $\{t_1, \ldots, t_m\} \cap (3idu_n, 3(i+1)du_n] = \emptyset$, then $\xi_i^{n,\kappa}$ and $(N_{3(i+1)du_n} - N_{3idu_n})$ are independent conditionally on $\mathcal F_i^{n,1}$, so we obtain that $\E[\xi_i^{n,\kappa} (N_{3(i+1)du_n} - N_{3idu_n})\,|\,\mathcal F_i^{n,1}] =0$.
If $\{t_1, \ldots, t_m\} \cap (3idu_n, 3(i+1)du_n] \neq \emptyset$, the fact that $\tilde f$ is bounded plus Lemma \ref{lemma: boundedness} imply that 
\[
\E[\,|\xi_i^{n,\kappa} (N_{3(i+1)du_n} - N_{3idu_n})|\,|\,\mathcal F_i^{n,1}] \le K \sqrt{u_n}.
\]
Since the intervals $(3idu_n, 3(i+1)du_n] $ are disjoint for different $i$, the number of such intervals having a non-empty intersection with $\{t_1, \ldots, t_m\} $ is bounded by $m$. Consequently, we end up with
\[
\sum_{i=0}^{[T/3du_n]-1} \E[\,|\xi_i^{n,\kappa} (N_{3(i+1)du_n} - N_{3idu_n})|\,|\,\mathcal F_i^{n,1}] \le mK\sqrt{u_n},
\]
which gives us \eqref{eqn: Jacod 5 b}. This completes the proof of Lemma \ref{lemma: Jacod} and therefore the proof of Theorem \ref{th2}.
\end{proof}

The proof of Proposition \ref{prop2} is somewhat simpler in comparison to the proof of Theorem \ref{th2}. Regarding Lemma \ref{lemma: 5.4}, part (i) of Proposition \ref{prop2} follows by showing the following lemma.

\begin{lem}
Assume (A1), (E). Let $r\in\{0,\ldots, d\}$, $\kappa, \kappa'=1,2$ and $g,h$ be any weight functions. Then, on $\Omega_T^{\le r}$, we have that 
\begin{multline} \label{eqn: Variance}
V(r,g,h)_T^{n,\kappa\kappa'}  \\ 
\toop 
\begin{cases}
3d \int_0^T \Gamma'_r(\sigma_s, \widetilde \sigma, v_s, b_s, \Sigma, g, 1) 
	+ \Gamma_r(\sigma_s, \widetilde \sigma, v_s, b_s, \Sigma, g, 1)^2 ds, &\text{if } \kappa= \kappa'=1, \\[1.5ex]
3d \int_0^T \Gamma'_r(\sigma_s, \widetilde \sigma, v_s, b_s, \Sigma, h, 2) 
	+ \Gamma_r(\sigma_s, \widetilde \sigma, v_s, b_s, \Sigma, h, 2)^2 ds, &\text{if } \kappa= \kappa'=2, \\[1.5ex]
3d \int_0^T \Gamma_r(\sigma_s, \widetilde \sigma, v_s, b_s, \Sigma, g, 1) 
	\Gamma_r(\sigma_s, \widetilde \sigma, v_s, b_s, \Sigma, h, 2) ds, &\text{if } \kappa=1, \kappa'=2.
\end{cases}
\end{multline}
\end{lem}

\begin{proof} 
Define the variables
\begin{equation*}
\rho(g,h)_i^{n,\kappa\kappa'} = 
\begin{cases}
\gamma_r(A(g)_i^{n,1}, B(g)_i^{n,1})^4, & \text{if } \kappa = \kappa'=1,\\
\gamma_r(A(h)_i^{n,2}, B(h)_i^{n,2})^4, & \text{if } \kappa = \kappa'=2,\\
\gamma_r(A(g)_i^{n,1}, B(g)_i^{n,1})^2 \gamma_r(A(h)_i^{n,2}, B(h)_i^{n,2})^2, & \text{if } \kappa = 1, \kappa'=2,
\end{cases}
\end{equation*}
which is the $i$th summand in the right-hand side of \eqref{eqn: V}. Define the variables
\[
\rho'(g,h)_i^{n,\kappa\kappa'}= \E[\rho(g,h)_i^{n,\kappa\kappa'}\,|\,\mathcal F_i^{n,1}], \qquad
\rho''(g,h)_i^{n,\kappa\kappa'} = \rho(g,h)_i^{n,\kappa\kappa'} - \rho'(g,h)_i^{n,\kappa\kappa'}.
\]
Using Lemma \ref{lemma: law}, we get that 
\[
\quad \rho'(g)_i^{n,\kappa} = \Gamma'_r(\sigma_i^{n,\kappa}, \widetilde \sigma, v_i^{n,\kappa}, b_i^{n,\kappa}, \Sigma^n, g^n, \kappa) + \Gamma_r(\sigma_i^{n,\kappa}, \widetilde \sigma, v_i^{n,\kappa}, b_i^{n,\kappa}, \Sigma^n, g^n, \kappa)^2.
\]
Just as in the proof of \eqref{eqn: Jacod 2 b} we can deduce that $9d^2 u_n \sum_{i=0}^{[T/3du_n]-1} \rho'(g,h)_i^{n,\kappa\kappa'}$ converges in probability to the right hand side of \eqref{eqn: Variance}. 
By construction, the sequence $(\rho''(g,h)_i^{n,\kappa\kappa'})_{i\ge0}$ is a $(\mathcal F_i^{n,1})$-martingale. Hence, we can use Doob's inequality and a calculation similar to the one in \eqref{Doob} to end up with
\[
9d^2 u_n \sum_{i=0}^{[T/3du_n]-1} \quad \rho''(g,h)_i^{n,\kappa\kappa'} \toop 0,
\]
which completes the proof of \eqref{eqn: Variance}.
\end{proof}

Part (ii) of Proposition \ref{prop2} essentially follows by the following lemma.

\begin{lem}
Assume (A1), (E). Let $r\in\{0,\ldots, d\}$ and $g,h$ be two weight function satisfying the conditions of Proposition \ref{prop1}(ii). Then, on $\Omega_T^{r}$, we have that
\begin{equation}\label{toop}
\frac{ \widehat R(g,h)_T^n - r}{\sqrt{u_n}} - \frac{U(r,g,h)_T^{n,1} - U(r,g,h)_T^{n,2}}{\log 2 \, S(r,g)_T^1} \toop 0.
\end{equation}
\end{lem}

\begin{proof}
Using the fact that $S(r,g)_T^1 = S(r,h)_T^2$, we obtain by an elementary calculation that, on $\Omega_T^r$,
\begin{equation*}
\widehat R(g,h)_T^n - r = \frac{\log(1+\sqrt{u_n} U(r,g,h)_T^{n,1}/S(r,g)_T^1) - \log(1+\sqrt{u_n} U(r,g,h)_T^{n,2}/S(r,g)_T^1)}{\log 2}.
\end{equation*}
Due to \eqref{clt} the sequence $U(r,g,h)_T^n$ is tight. By a Taylor expansion, one obtains that $\log(1+x) = x + O(x^2)$ for $|x|<1$, so we get (for $n$ sufficiently large) 
\[
\log(1+\sqrt{u_n} U(r,g,h)_T^{n,\kappa}/S(r,g)_T^1) = \sqrt{u_n} U(r,g,h)_T^{n,\kappa}/S(r,g)_T^1) + O_{\PP}(u_n).
\]
This readily implies \eqref{toop}.
\end{proof}
The continuous mapping theorem for stable convergence then implies that, on $\Omega_T^{r}$,
\begin{equation}\label{stable}
\frac{U(r,g,h)_T^{n,1} - U(r,g,h)_T^{n,2}}{\log 2 \, S(r,g)_T^1} \stab \frac{\mathcal U''(r,g,h)_T^{1} - \mathcal U''(r,g,h)_T^{2}}{\log 2 \, S(r,g)_T^1},
\end{equation}
where $\mathcal U''(r,g,h)_T$ is the limit in \eqref{clt} (see also equation \eqref{limit}). The right-hand side of \eqref{stable} is mixed normal with $\mathcal G$-conditional mean 0 and $\mathcal G$-conditional variance
\begin{equation*}\label{conditional variance}
\frac{3d \int_0^T \Gamma'_r(\sigma_s, \widetilde{\sigma}, v_s,b_s, \Sigma ,  g,1) ds +
3d \int_0^T \Gamma'_r(\sigma_s, \widetilde{\sigma}, v_s,b_s, \Sigma ,  h,2) ds}{ (S(r,g)_T^1 \, \log 2)^2}>0.
\end{equation*}
The positivity of the variance is a consequence of \eqref{eqn: positive} in Lemma \ref{lemma:Gamma}.
At this stage, \eqref{finalclt} follows by part (i) of Proposition \ref{prop2}, Theorem \ref{th1} and the delta method for stable convergence.
\qed

\subsection{Proof of Corollary \ref{corollary: testing}} \label{sec5.5}

The implication at \eqref{test 1a} is a direct consequence of the stable convergence at \eqref{finalclt}. To prove the consistency at \eqref{test 1b}, it is sufficient to show that for any $r'\neq r$ we have that
\[
\mathbb P(\mathcal{C}_{\alpha}^{n, =r} \cap \Omega_T^{r'}) \rightarrow \mathbb P(\Omega_T^{r'}).
\]
Let $\Phi$ be the right-hand side of \eqref{finalclt}. Then we have by Proposition \ref{prop2}(ii) that 
\[
\mathbb P(\mathcal{C}_{\alpha}^{n, =r} \cap \Omega_T^{r'}) - \widetilde{\mathbb{P}} \Big( \big\{ \big|\Phi + \frac{r'-r}{\sqrt{u_n V(n,T,g,h)}}\big|> z_{1-\alpha/2} \big\} \cap \Omega_T^{r'} \Big) \to0.
\]
By Proposition \ref{prop2}, Theorem \ref{th1} and Lemma \ref{lemma:Gamma}, $V(n,T,g,h)$ converges in probability to a positive-valued limit, such that $u_n V(n,T,g,h) \toop 0$ and hence
\[
\widetilde{\mathbb{P}} \Big( \big\{ \big|\Phi + \frac{r'-r}{\sqrt{u_n V(n,T,g,h)}}\big|> z_{1-\alpha/2} \big\} \cap \Omega_T^{r'} \Big) \to \mathbb P(\Omega_T^{r'}),
\]
which shows \eqref{test 1b}.
To show \eqref{test 2a}, let $A\subset \Omega_T^{\le r}$ with $\mathbb P(A)>0$. Then we obtain
\begin{multline*}
\mathbb P(\mathcal{C}_{\alpha}^{n, \le r} \,|\,A) 
= \sum_{r'\le r} \mathbb{P}(\mathcal{C}_{\alpha}^{n, \le r} \cap \Omega_T^{r'} \,|\,A)
\le  \sum_{r'\le r} \mathbb P(\mathcal{C}_{\alpha}^{n, \le r'} \cap \Omega_T^{r'} \,|\,A) \\
\to\sum_{r'\le r} \widetilde{\mathbb{P}} (\{\Phi > z_{a-\alpha} \} \cap \Omega_T^{r'} \,|\,A)
= \alpha \mathbb{P}( \Omega_T^{\le r}\,|\,A) = \alpha.
\end{multline*}
We essentially used the convergence at \eqref{finalclt} as well as the fact that $\Phi$ is independent of $\mathcal G$ and $\Omega_T^{r'} \in \mathcal G$.
The consistency result at \eqref{test 2b} follows in the same manner as \eqref{test 1b}.
\qed

\end{document}